\documentclass{amsart}
\pdfoutput=1
\usepackage{amsmath,amsfonts,amsthm,graphicx}
\usepackage{hyperref}
\usepackage{comment}
\usepackage{graphicx,psfrag}  
\usepackage[psamsfonts]{amssymb}
\usepackage{amscd}
\usepackage{color}
\usepackage[all,cmtip]{xy}
\usepackage{mathtools}

\usepackage{tikz}

\title[Coexistence of Simple Spectrum]{Coexistence of Measures with Simple Lyapunov Spectrum for Fiber Bunched Cocycles}
\author{Jonathan DeWitt and Daniel Mitsutani}
\address[Mitsutani]{Department of Mathematics, the University of Chicago, Chicago, IL, USA, 60637}
\email{mitsutani@math.uchicago.edu}
\address[DeWitt]{Department of Mathematics, the University of Maryland, College Park, MD, USA, 20742}
\email{dewitt@umd.edu}

\date{\today}
\theoremstyle{plain}

\newtheorem{theorem}{Theorem}[section]
\newtheorem{thmy}{Theorem}

\newtheorem{proposition}[theorem]{Proposition}
\newtheorem{lemma}[theorem]{Lemma}

\newtheorem{claim}[theorem]{Claim}
\newtheorem{definition}[theorem]{Definition}
\theoremstyle{remark}
\newtheorem{remark}[theorem]{Remark}
\newtheorem{example}[theorem]{Example}

\def\hx{\hat x}
\def\hSigma{\hat \Sigma}
\def\hsigma{\hat \sigma}
\def\hm{\hat m}
\def\hq{\hat q}
\def\heta{\hat \eta}
\def\hy{\hat y}
\def\hmu{\hat \mu}
\def\hnu{\hat \nu}

\def\hz{\hat z}

\def\title{\em}

\def\id{\hbox{id}}

\def\hA{\hat{A}}
\def\hSigma{\hat{\Sigma}}
\def\hsigma{\hat{\sigma}}
\def\hmu{\hat{\mu}}
\def\hm{\hat{m}}
\def\hx{\hat{x}}

\def\cL{\mathcal{L}}

\def\cS{\mathcal{S}}

\def\P{\mathbb{P}}

\newcommand{\Gr}{\operatorname{Gr}}

\newcommand{\SL}{\operatorname{SL}}
\newcommand{\GL}{\operatorname{GL}}
\newcommand{\Sp}{\operatorname{Sp}}

\newcommand{\abs}[1]{\left| #1 \right|}
\newcommand{\mc}[1]{\mathcal{ #1 }}
\newcommand{\wt}[1]{\widetilde{ #1 }}

\def\hp{\hat p}

\def\hB{\hat{B}}

\def\hsigma{\hat{\sigma}}

\def\transverse{\,\raise2pt\hbox to1em{\hfil$\top$\hfil}\hskip -1em \hbox
to1em{\hfil$\cap$\hfil}\,} 

\newcommand\R{\mathbb R}

\newcommand\N{{\mathbb N}}

\newlength{\figboxwidth} 

\begin{document}

\begin{abstract} 

We prove that if a Hölder continuous fiber-bunched cocycle $\hA$ over an invertible hyperbolic transitive shift $\hSigma$ satisfies an appropriate strong irreducibility condition on Grassmannians, then $\hSigma$ admits an ergodic measure $\hmu$ with full support and product structure with simple Lyapunov spectrum if and only if any other ergodic measure with full support and product structure also has simple Lyapunov spectrum.
\end{abstract}

\maketitle

\section{Introduction}
Lyapunov exponents describe the exponential growth or decay of norms of products of matrices which in many settings are used to determine different ergodic and dynamical properties of the system. The question of existence of a non-zero exponent has been widely studied with such applications in mind. A further question is whether all Lyapunov exponents have multiplicity 1, referred to as simple Lyapunov spectrum, which provides more detailed information on ergodic averages, and implies in many settings the existence of one dimensional invariant foliations for the system. 
	
The study of criteria for non-zero exponents and simple spectrum for iid random products of matrices traces back to works of Furstenberg \cite{furstenberg1963noncommuting}, Guivarch and Raugi \cite{guivarch1984random}, and Goldsheid and Margulis \cite{Goldsheid1989lyapunov}, among many others. In all of these works, a common feature is an irreducibility hypothesis on the set of matrices generating the random products which ensures that cancellations of the expanding and contracting directions of the matrices do not occur almost surely. Under such irreducibility hypotheses, their theorems give existence of a non-zero Lyapunov exponent \cite{furstenberg1963noncommuting} or simple spectrum \cite{guivarch1984random}, provided that the matrices generating the products satisfy necessary contraction and expansion properties. 

Motivated by applications for partially hyperbolic systems and holomorphic foliations, these criteria were later extended (\cite{bonatti2003},  \cite{bonattiviana}) and refined (\cite{avila2007simplicity}, \cite{avila2007simplicity1}) to apply to the setting of fiber-bunched matrix cocycles, where the product of matrices is formed as a skew product over a driving hyperbolic dynamical system, giving non-trivial and simple Lyapunov spectrum with respect to any measure satisfying a natural assumption known as local product structure. Fiber-bunching is a partial hyperbolicity-type condition which ensures the existence of holonomies for the cocycle, which in turn allows for many of the random features from the iid setting to apply in this setting.

In these works, the criteria found for non-trivial (i.e. not all zero) or simple (i.e. all distinct) Lyapunov exponents is a concrete setup, expressed as a transversality condition, which provides  \textit{sufficient} criterion for the existence of non-zero exponents \cite{bonatti2003} and simple spectrum \cite{bonattiviana}, respectively.  These criteria in turn have been used to prove simplicity of spectrum generically for: fiber-bunched cocycles over hyperbolic systems \cite{avila2007simplicity}, certain cocycles over partially hyperbolic systems \cite{poletti2019simple}, and for fiber-bunched Anosov flows and geodesic flows in pinched negative curvature \cite{mitsutani}. It has also been applied to specific situations such as in the solution of the Kontsevich-Zorich conjecture in Teichmuller dynamics \cite{avila2007simplicity1}. 

All criteria found so far for simple spectrum with respect to all reasonable measures in the fiber-bunched setting so far have only been sufficient criteria, in the form of existence a particular setup over a periodic point. The goal of this paper is to provide a criterion for fiber-bunched cocycles over hyperbolic systems to satisfy simple  Lyapunov spectrum with respect to a natural class of measures under an irreducibility hypothesis in the spirit of Guivarch and Raugi \cite{guivarch1984random}. The advantage of this approach is that an irreducibility hypothesis is evidently a \textit{necessary} condition to guarantee that all measures have simple Lyapunov spectrum for a non-trivial reason.

Our main result shows that an irreducible fiber-bunched cocycle with one reasonable measure with simple Lyapunov spectrum in fact must have simple Lyapunov spectrum with respect to all reasonable measures.

\begin{thmy} \label{theorem: simplespectrum}
Let $\hat{\Sigma}$ be a transitive invertible subshift of finite type and $\hat{A}\colon \hSigma\to \SL(d,\R)$ a fully irreducible H\"older continuous fiber bunched cocycle. 
 Suppose there exists an ergodic, fully supported, invariant measure $\hmu$ with continuous local product structure on $\hat{\Sigma}$ with simple Lyapunov spectrum. Then any other ergodic, fully supported, invariant measure $\heta$ with continuous local product structure on $\hat{\Sigma}$ also has simple Lyapunov spectrum.
\end{thmy}
The complete definitions in the theorem above are provided in Section \ref{sec:defs}. The condition \textit{full irreducibility} is defined in Section \ref{sec:smoothness}. For fiber-bunched cocycles, it is the analogue on all Grassmannians $\Gr(k,d)$, $1\le k\le d$, of strong irreducibility for iid~random products of matrices. We remark that full irreducibility is a strictly weaker hypothesis than the transversality condition required in \cite{bonatti2003} and other works, which is often referred to as twisting. 

 We will actually prove a finer result than Theorem \ref{theorem: simplespectrum}. Theorem \ref{theorem: simplespectrum} is proved by  individually showing that each Lyapunov exponent must have multiplicity 1. In fact, its proof is immediate from:
{\color{red} }
\begin{thmy} \label{theorem: maintheorem}
Let $1 \leq k \leq d-1$, $\hat{\Sigma}$ a transitive subshift of finite type and $\hat{A}\colon \hSigma\to \SL(d,\R)$ a H\"older continuous fiber bunched cocycle strongly irreducible on $\Gr(k,d)$ and on $\Gr(d-k, d)$.

Suppose there exists an ergodic, fully supported, invariant measure $\hmu$ on $\hat{\Sigma}$ with continuous local product structure  and $\lambda_k(\hmu) > \lambda_{k+1}(\hmu)$. Then any other ergodic, fully supported, invariant measure $\heta$ with continuous local product structure on $\hat{\Sigma}$ also has $\lambda_k(\heta) > \lambda_{k+1}(\heta)$.
\end{thmy}
The precise definition of \textit{strong irreducibility on} $\Gr(k,d)$ is given in Section \ref{sec:smoothness}. Full irreducibility is defined as strong irreducibility on $\Gr(k,d)$ for each $1 \leq k \leq d-1$ so it is clear that Theorem \ref{theorem: simplespectrum} follows from Theorem \ref{theorem: maintheorem}.

\begin{remark} Although requiring a condition on $\Gr(k,d)$ and $\Gr(d-k,d)$ might seem unexpected, one may compare with \cite[Thm.~1]{bonattiviana}, which under a certain twisting condition establishes simultaneously that both the top and bottom Lyapunov exponent are one dimensional.
\end{remark}

 In \cite[Thm.~1]{bonatti2003}, Bonatti, Gomez-Mont, and Viana prove an alternative for fiber-bunched cocycles: either every equilibrium state has at least two distinct Lyapunov exponents or there exists a holonomy invariant, cocycle invariant measure on its projective extension. This, in turn, should imply either a failure of irreducibility or conformality of the cocycle (which implies that all exponents are $0$ for all invariant measures). For the sake of completeness, we include this statement in the following theorem:

\begin{thmy}\label{thm:furstenberg_analog}
Suppose that $\hat{\Sigma}$ is a transitive hyperbolic system and that $\hat{A}\colon \Sigma\to \SL(d,\R)$ is a fiber-bunched H\"older continuous cocycle that is strongly irreducible on $\Gr(1,d)$. Then exactly one of the following holds:
\begin{enumerate}
\item
For every measure $\hat{\mu}$ with full support and continuous product structure, $\hat{\mu}$ has two distinct Lyapunov exponents, or;
\item
$\hat{A}$ preserves a H\"older continuous conformal structure, and thus every invariant measure has all Lyapunov exponents equal to $0$.
\end{enumerate}
\end{thmy}

 In view of this theorem, Theorem \ref{theorem: maintheorem} above is a version of this result which extends to the study of every Lyapunov exponent, and thus to the study of simplicity of Lyapunov spectrum.

 \subsection{Prior Results.} To better frame the results of this paper, we now discuss in some detail some of the earlier results mentioned above. 
 We begin with a review of the case of iid matrix products. Throughout, we let  $(A_1,\ldots,A_m)\subseteq \SL(d,\R)$, $d\ge 2$ and let $G$ denote the semigroup generated by the matrices $A_1,\ldots,A_m$.

The foundational result on the existence of non-zero Lyapunov exponents is the theorem of Furstenberg \cite{furstenberg1963noncommuting}: if $G$ is not compact and $G$ acts strongly irreducibly on $\R^d$, that is, $G$ does not preserve a non-trivial finite union of subspaces of $\R^d$, then almost surely the iid uniformly distributed random products of the matrices in $A_1,\ldots,A_m$ have two distinct Lyapunov exponents. Guivarch and Raugi \cite{guivarch1984random} improve on this theorem to obtain a criterion for simplicity of the top Lyapunov exponent. They prove that if  $G$ acts strongly irreducibly on $\R^d$ and there exists a sequence of elements $g_n\in G$ such that $\|g_n\|^{-1}g_n$ converges to a rank $1$ matrix, then $\lambda_1>\lambda_2$. This implies a related result for $\lambda_k>\lambda_{k+1}$ by considering the associated action on $\Lambda^k(\R^d)$. Finally, Goldsheid and Margulis \cite{Goldsheid1989lyapunov} improve this theorem to only have a hypothesis on the algebraic closure $X$ of $G$ in $\SL(d,\R)$. All of the theorems above can be seen as presenting an alternative under an irreducibility hypothesis. 

 In \cite{avila2007simplicity1}, the Lyapunov exponents of a locally constant cocycle over a shift equipped with a measure having a type of product structure is considered. They show that this cocycle has simple Lyapunov spectrum if it is pinching and twisting. Here, pinching and twisting reduce to an irreducibility condition on $k$-planes as well as a condition that says one can separate the singular value of matrices in the sense that for all $K>0$, one can find an element $A$ the cocycle acts by such that $\sigma_k(A)>K\sigma_{k+1}(A)$. In \cite{cambrainha2016typical}, it is shown that almost every locally constant symplectic cocycle has simple spectrum.

Outside of the locally constant setting, it is more difficult to obtain results about Lyapunov exponents of cocycles. We describe here work studying fiber bunched cocycles over hyperbolic bases. In \cite{bonatti2003}, those authors show that there is an open and dense subset of the space of fiber bunched cocycles such that for every equilibrium state there are two distinct Lyapunov exponents by identifying a particular obstruction to hyperbolicity: a family of $f$-invariant, holonomy invariant probability measures on every fiber of the bundle $\Sigma\times \mathbb{RP}^{d-1}$ that the cocycle acts on. Note that the result of \cite{bonatti2003} does not offer any information about which two Lyapunov exponents might be distinct nor their multiplicity. See also \cite[Theorem B]{avila2013holonomy}, where a similar result is obtained for cocycles over partially hyperbolic systems with $0$  center exponents. In \cite{viana2008almost}, Viana shows, without a fiber bunching assumption, that for a fixed ergodic measure with product structure almost all H\"older continuous cocycles have two distinct Lyapunov exponents.

In the fiber bunched setting, a condition known as \emph{pinching and twisting} is sufficient for a fiber bunched cocycle to have simple spectrum. In \cite{bonattiviana}, it is shown that a generic fiber bunched cocycle over a hyperbolic system has simple Lyapunov spectrum. This paper is one of the first where a type of pinching and twisting for fiber bunched cocycles appears. In \cite{avila2007simplicity}, a simpler formulation of the pinching and twisting condition is given.  Pinching and twisting has seen great application in the study of simplicity of Lyapunov spectrum. In \cite{backes2020simplicity}, it is shown that one can adapt the pinching and twisting criterion of Avila and Viana in order to show that, in a particular sense, generic cocycles over non-uniformly hyperbolic systems have simple Lyapunov spectrum. In \cite{Bessa2018positivity}, a version of \cite{viana2008almost} for cocycles taking value in a semi-simple Lie group is proved, which says that a generic cocycle has two distinct Lyapunov exponents even when the cocycle lies, for example, in $\Sp(n,\R)$. In \cite{poletti2019simple}, the twisting and pinching criterion is extended to apply to cocycles over certain partially hyperbolic diffeomorphisms.

\subsection{Discussion}
 When studying an abstract hyperbolic system, there is not always an obvious choice of measure that is the correct one to equip the system with. However, from the point of view of the thermodynamic formalism, equilibrium states provide the most natural choice of a measure. Consequently, it is conceptually satisfying to find criteria that apply to all these measures simultaneously. 


On a technical level, this paper develops an approach to studying the Lyapunov exponents of fiber bunched cocycles that uses a different sort of ``twisting" condition.  In \cite{avila2007simplicity1}, the twisting condition requires that the cocycle holonomy around a particular homoclinic loop satisfies a specific condition that it can move subspaces off of themselves. This condition is localized to a very specific periodic point $\hp$. In our setup, we have a measure $\hmu$ that has simple spectrum and another measure $\heta$ that we're interested in showing has simple spectrum. By using the Osceledec subspaces of $\hmu$ we are able to obtain a sort of non-uniform twisting condition that is not localized near any particular point of $\hat{\Sigma}$. The Osceledec subspaces associated to $\hmu$ are, in a certain sense, randomly distributed along local stable and unstable sets (See Subsection \ref{subsec:distribution_of_osceledec_subspaces}). Our main technical contribution is figuring out how to use this feature in order to show convergence to a Dirac measure in Section \ref{sec: dirac}.

A secondary technical point is made in this paper: we have formulated and proved our results by using irreducibility conditions stated for Grassmannians $\Gr(k,d)$ rather than conditions on exterior powers $\Lambda^k(\R^d)$. Although the difference is minor, we find that working with Grassmannians is conceptually cleaner and more in keeping with the approach pursued by Viana in his many works on Lyapunov exponents.

It remains interesting to determine if one can weaken the hypothesis that ``pinching" occurs for a fully supported measure with continuous product structure to the existence of a single pinching orbit.

\subsection{Structure of Paper} Section \ref{sec:defs} contains all necessary definitions and some standard constructions. In Section \ref{sec:smoothness}, we state our irreducibility hypothesis for fiber-bunched cocycles and show that it implies the random distribution of unstable subspaces as well as a corresponding statement about $u$-states. Finally, Sections \ref{sec: dirac} and \ref{sec: proof} contain the proof of Theorem \ref{theorem: maintheorem} and Section \ref{sec: furstenberg} contains the proof of Theorem \ref{thm:furstenberg_analog}. 

\vspace{.5cm}

\noindent\textbf{Acknowledgments.} The authors thank their advisor, Amie Wilkinson, for suggesting that they study Lyapunov exponents together, which lead to the results in this paper. The authors also thank Aaron Brown for helpful discussions. This material is based upon work supported by the National Science Foundation under Award No.~DMS-2202967.

\section{Background and Definitions}\label{sec:defs}

For $m\ge 2$, we let $M$ be an irreducible $m\times m$ matrix with coefficients in $\{0,1\}$ and let $\hsigma\colon \hSigma\to \hSigma$ be the associated two-sided left  shift. Below, we always assume that $\hSigma$ is of this form. One can formulate an abstract notion of a hyperbolic system that generalizes the dynamical properties of this system, see for example the setting described in \cite{kalinin2011livsic}, but we work with shifts here. The space $\hSigma$ is endowed with a metric given for $\hx\neq \hy$ by
\[
d(\hx,\hy)\coloneqq
2^{-N(\hx,\hy)}\text{ where } N(\hx,\hy)=\min\{\abs{n}:\hx_n\neq \hy_n\}.
\]
We will use the usual notation for cylinder sets where $[0;i]$ denotes the cylinder with zeroth symbol $i$. We write $\Sigma$ for the one sided left-shift associated to $\hat{\Sigma}$.

We define the local stable and unstable sets of a point $\hx\in \hSigma$:
\[\begin{aligned}
&W^s_{loc}(\hx)&=\{\hy\in \Sigma: \hx_i=\hy_i, \text{ for }i\ge 0\},\\
&W^u_{loc}(\hx)&=\{\hy\in \Sigma: \hx_i=\hy_i, \text{ for } i\le 0\}.
\end{aligned}
\]
We let $\pi_s$ and $\pi_u$ be the obvious projections onto the local stable and unstable sets. The global stable and unstable sets are defined similarly for $\hx\in \hSigma$:
\[\begin{aligned}
&W^s(\hx)=\{\hy\in \Sigma: \exists n \geq 0, \, \hx_i=\hy_i, \text{ for }i\ge n\},\\
&W^u(\hx)=\{\hy\in \Sigma: \exists n \leq 0,\hx_i=\hy_i, \text{ for } i\le n\}.
\end{aligned}
\]

For two points $\hx,\hy$ in the same cylinder $[0;i]$ we can take their Smale bracket $[\hx,\hy]$ which is the unique point in $W^u_{loc}(\hx)\cap W^s_{loc}(\hy)$. Symbolically this point is given by the sequence:
\[
(\ldots,\hx_{-2},\hx_{-1},\hx_0,\hy_1,\hy_2,\ldots).\]

We will also work with one sided shift spaces $\Sigma^s$ and $\Sigma^u$. These are defined by
\begin{align}
\Sigma^u=\{(x_n)_{n\ge 0}: q_{x_n,x_{n+1}}=1\text{ for all } n\ge 0\},\\
\Sigma^s=\{(x_n)_{n\le 0}: q_{x_n,x_{n+1}}=1\text{ for all } n<0\},
\end{align}
with the corresponding left and right shifts $\sigma_u$ and $\sigma_s$. We also have notation for cylinders in $\hat{\Sigma}$:
\begin{align}
[m;a_0,\ldots,a_k]=\{\omega\in \hat{\Sigma}: \omega_{m+i}=a_{m+i}, 0\le i\le k\}.
\end{align}
Analogously we define $[m;a_0,\ldots,a_k]^u$ and $[m;a_0,\ldots,a_k]^s$ as notation for cylinders in $\Sigma^u$ and $\Sigma^s$.  Since many arguments are symmetric for the left and right shift, whenever we need to consider only one of them at a time, say the left shift, in keeping with the literature we will also write $\Sigma^u$ as $\Sigma$ and write $P$ for the projection $P\colon \hat{\Sigma}\to \Sigma$. Thus symbols with hats are used for two-sided objects and symbols without hats for the corresponding one-sided objects on the left shift $\Sigma^u$.

We let $\hA\colon \hSigma\to \GL(d,\R)$ be an $\alpha$-H\"older continuous function, we think of $\hA$ as defining a cocycle $\hat{\mc{A}}\colon \hSigma\times \R^d\to \hSigma\times \R^d$ in the usual way and make the standard abuse of calling $\hat{A}$ a cocycle. We define for $n\in \N$,
\[
\hA^n(\hx)=\hA(\hsigma^{n-1}(\hx))\cdots \hA(\hx),
\]
and for $n\in \N$ also define $\hA^{-n}(\hx)=(\hA^n(\hsigma^{-n}\hx))^{-1}$ as is standard.

In this paper we study fiber bunched cocycles.
\begin{definition}
We say that an $\alpha$-H\"older cocycle $\hA\colon \hSigma\to \GL(d,\R)$ is \emph{fiber-bunched} if for every $\hx\in \hSigma$:
\[
\|\hA(\hx)\|\|\hA(\hx)^{-1}\|2^{-\alpha}<1.
\]
\end{definition}
Fiber bunched cocycles have a number of nice properties. The most important of which is that the cocycle has holonomies. See, for example, \cite[Sec.~4.2]{sadovskaya2015cohomology}.
\begin{proposition}\label{prop:holonomies_exist}
    Suppose that $\hA$ is a fiber-bunched cocycle over $\hsigma$. Then for any $\hx\in W^s_{loc}(\hy)$,
    \[
H^s_{\hx\hy}\coloneqq \lim_{n\to \infty} \hA^{-n}(\hy)\hA^n(\hx)
    \]
    exists and $H^s_{\hx\hy}$, called the stable holonomy defines a H\"older continuous function from $x,y$ to $\GL(d,\R)$. The local stable holonomy satisfies:
    \begin{enumerate}
\item
$H^s_{\hx\hx}=\id$ and $H^s_{\hy\hz}\circ H^s_{\hx\hz}=H^s_{\hx\hz}$ for any $\hy,\hz\in W^s_{loc}(\hx)$.
\item
$\hA(\hx)=H^s_{\hsigma(\hy) \hsigma(\hx)}\circ \hA(\hy)\circ H^s_{\hx\hy}$
    \end{enumerate}   
    The analogous statement holds for the local unstable sets and their holonomies.
\end{proposition}

We will also make use of a \emph{global stable holonomy}. For a point $\hx\in W^s(\hy)$ this is defined as follows. Iterate the cocycle until $\hsigma^n(\hx)\in W^s_{loc}(\hsigma^n(\hy))$, then using this $n$, set
\begin{equation}\label{eqn:global_holonomy_def}
H^s_{\hx\hy}=\hA^{-n}(\hy)H_{\hsigma^n(\hx)\hsigma^n(\hy)}^s\hA^n(\hx),
\end{equation}
where the holonomy on the right hand side is given by Proposition \ref{prop:holonomies_exist}.

 Using holonomies, via a construction which we recall now from \cite{avila2007simplicity}, we may assume that the cocycle is constant along local stable sets, so that $\hA$ induces a cocycle over $\Sigma^u = \Sigma$. Fix some arbitrary point $x^- \in \Sigma^s$, the left-sided shift, and for each $\hat{x} \in \hSigma$ let $\phi^u(\hx)$ be $W^u_{loc}(x^-) \cap W^s_{loc}(\hx)$. Then $A(\hx)$, the \textit{reduced cocycle}, is defined by
\begin{equation}\label{eqn:def_reduced_cocycle}
 A(\hx) = H^s_{\hsigma \hx, \phi^u(\hsigma \hx)} \circ \hat{A}(\hx) \circ	H^s_{\phi^u(\hx) \hx } = H^s_{\hsigma(\phi^u(\hx)) \phi^u(\hsigma(\hx))} \circ \hat{A}(\phi^u(\hx))
 \end{equation}
and $\hat{x} \to A(\hat{x})$ is easily seen to be constant along local stable manifolds due to (2) in Proposition \ref{prop:holonomies_exist}.

\subsection{Lyapunov Exponents} \label{ssec: lyapunov exponents} For $\hmu$ any shift-invariant measure on $\hSigma$ and $\hA: \hSigma \to \GL(d,\R)$ any measurable cocycle such that $\int_{\hSigma} \log^+ \|\hA^{\pm 1}\|\, d\hmu < \infty$, where $\log^+(x)=\max(0,x)$, the Oseledec' theorem guarantees the existence of numbers 
$$\lambda_1(\hmu) \geq \lambda_2(\hmu) \geq \dots \geq \lambda_d(\hmu)$$
and an $\hA$-invariant measurable splitting $\R^d = E_{d_1}(\hx) \oplus \dots \oplus E_{d_k}(\hx)$ defined $\hmu$-a.e.~($\dim E_{d_i} = d_i - d_{i-1}$ for $i = 2, \dots k$ and $\dim E_1 = d_1$), where the numbers $d_i$ correspond to the indices such that $\lambda_{d_i}(\hmu) > \lambda_{d_i+1}(\hmu)$, and such that for $v \in E_{d_i}(\hx)$:
\[
\lim_{n \to \pm \infty} \frac{1}{n} \log \|\hA^{n}(\hx) v\| = \lambda_{d_i}(\hmu).
\]
Moreover, for $\hmu$-a.e.~$\hx$ as above, we define the unstable flag 
\[
\{0\} \subset E^u_{d_1}(\hx) \subset E^u_{d_2}(\hx) \subset \dots \subset E^u_{d_k}(\hx) = \R^d
\]
by $E^u_{d_i}(\hx) = E_{d_1}(\hx) \oplus \dots \oplus E_{d_i}(\hx)$ and the stable flag $E^s$ is defined analogously. 
From the proof of the one-sided Oseledec' theorem, the unstable flag depends only on the past of $\hx$, i.e., on $\hA^{-n}(\hx)$ for $n \geq 0$ and similarly the stable flag which only depends the future of $\hat{x}$, i.e.~on $\hA^{n}(\hx)$ for $n \geq 0$. For a point $\hat{x}$ that has an Osceledec splitting as defined above, we will write $E^u_k(\hat{x})$ for the $k$-dimensional asymptotically most expanded subspace at $\hat{x}$ when this is well defined. In the case of simple spectrum, this is well defined as long as $\hat{x}$ is Lyapunov regular because $E^u_k(\hat{x})$ is then the direct sum of the top $k$ subspaces $E_1(\hat{x}),\ldots ,E_k(\hat{x})$. Similarly we define $E^s_{k}(\hat{x})$ to be the $k$-dimensional subspace that is most contracted asymptotically at $\hat{x}$.

To make this statement precise we introduce the following notation. For a matrix $A \in \GL(d,\R)$ order its singular values non-increasingly and write them $\sigma_1(A) \ge \sigma_2(A) \ge \cdots \ge \sigma_d(A)$. If $\sigma_k(A) > \sigma_{k+1}(A)$, then from the singular value decomposition there exists a unique $(d-k)$-dimensional subspace least expanded by $A$, which we denote by $S_{d-k}(A)$. Similarly, there exists a unique $k$-dimensional subspace most expanded by $A$, and we denote its \textit{image under} $A$ by $U_k(A)$. For $\hx \in \hSigma$, let $U_k(\hx, n) := U_k(\hA^n(\hsigma^{-n}(\hx)))$ and $S_{d-k}(\hx, n) := S_{d-k}(\hA^n(\hx))$ when $\sigma_k(A^n(\hx)) > \sigma_{k+1}(A^n(\hx))$. Observe that if $\lambda_k(\hmu) > \lambda_{k+1}(\hmu)$, then for $\hmu$-a.e.\ $\hx$ there is some $n_0 > 0$ such that for all $n \geq n_0$ the subspaces $U_k(\hx, n)$ and $S_{d-k}(\hx, n)$ are well defined. Then by the usual proof of Oseledec' theorem, for $\hmu$-a.e.\ $\hx$,
\begin{equation}
    \label{eq: convergenceofsubspaces} \lim_{n \to \infty} U_k(\hx, n) = E^u_k(\hx), \hspace{1cm}\lim_{n \to \infty}  S_{d-k}(\hx, n) = E^s_{d-k}(\hx).
\end{equation}

In the fiber-bunched setting, once we have passed to the reduced cocycle the unstable flag depends only on the past of $\hx$.
\begin{lemma}\label{lem:unstable_bundles_cts_along_unstable_leaves}
Let $1 \leq k \leq d$. Suppose $E^u_k(\hx)$ exists for some $\hx \in \hSigma$ in the sense that the limits in \eqref{eq: convergenceofsubspaces} exist. Then for every $\hy \in W^u_{loc}(\hx)$, $E^u_k(\hy)$ exists and is given by:
$$E^u_k(\hy) = H^u_{\hx \hy} E^u_k(\hx).$$
In particular $E^u_k$ is uniformly continuous along local unstable manifolds where it is defined. The analogous statement holds for $E^s_k$ with respect to $H^s$ holonomies.
\end{lemma}
\begin{proof} We write the proof in the case of the stable bundle to avoid negative signs. The point is that for $\hat{x}\in W^s_{loc}(\hat{y})$, we have that
\begin{equation}\label{eqn:intertwine_honomomies_B}
H^s_{\hsigma^n(\hat{x})\hsigma^n(\hat{y})}A^n(\hat{x})=A^n(\hat{y})H_{\hat{x}\hat{y}}^s.
\end{equation}
and $H^s_{\sigma^n(\hat{x})\sigma^n(\hat{y})}\to \text{Id}$ while the gaps between the singular values of $A^n(\hat{x})$ are growing arbitrarily large. Writing $S_k(A)$ for the span of the smallest $k$ singular vectors of a matrix $A$. From \eqref{eqn:intertwine_honomomies_B} we have:
\[
S_k(H^s_{\hat{\sigma}^n(\hat{x})\hat{\sigma}^n(\hat{y})}A^n(\hat{x}))=S_k(A^n(\hat{y})H^s_{\hat{x}\hat{y}}).
\]
By directly estimating the norm restricted to subspaces by using the angles between singular subspaces of $A(\hat{x})$, one can deduce that $S_k(H^s_{\sigma^n(\hat{x})\sigma^n(\hat{y})}A^n(\hat{x}))$ converges to $ E^s_k(\hat{x})$ and   $S_k(A^n(\hy)H^s_{\hat{x}\hat{y}})$ converges to $H^s_{\hat{y}\hat{x}}(E^s_k(\hat{y}))$. 
Thus by passing to the limit we find that 
\[
E^s_k(\hat{x})=H^s_{\hat{y}\hat{x}}(E^s_k(\hat{y})),
\]
as desired.
\end{proof}



\subsection{Product Structure of Measures}
We now describe continuous local product structure. For each $1\le i\le m$, there exists the projection $[0;i]\to \pi_s([0;i])\times \pi_u([0;i])$. For convenience, we will denote by $\hmu^s_i,\hmu^u_i$ these pushforwards of $\mu\vert_{[0;i]}$ by $\pi_s$ and $\pi_u$ on the one sided shifts. We say that $\hat{\mu}$ has \emph{continuous local product structure} if for each symbol $i$ there is a continuous function $\psi\colon \hSigma\to (0,\infty)$ such that 
\begin{equation}\label{eqn:definition_of_product_structure}
\hmu\vert_{[0;i]} =\psi\cdot (\hmu^s_{i}\times \hmu^u_{i}).
\end{equation}

 Equation \eqref{eqn:definition_of_product_structure} is a description of the conditional measures on local stable and unstable leaves:
\begin{equation}\label{eqn:integral_against_phi_psi}
\int_{[0;i]} \phi\,d\hmu=\int_{[0;i]^u} \left(\int_{[0;i]^s} \phi \psi \,d\hmu^s_i \right)d\hmu^u_i.
\end{equation}

Note that local unstable holonomy between local stable sets has a bounded Jacobian for these conditional measures, where the bound is given by $0<\max \psi/\min\psi<\infty$. In particular, from Equation \eqref{eqn:integral_against_phi_psi} we have the following description of the conditional measures along local stable sets. On the local stable set containing the point $\hx \in \hat{\Sigma}$, we may take
\[
\hmu^s_{\hx}=\psi(\hy)\hmu^s_i\text{ for } \hy\in W^s_{loc}(\hx).
\]

We now consider how the stable measures transform under application of the dynamics. Our goal is to show that we may arrange that the conditional measures along unstable leaves to not become distorted at small scale as we iterate the dynamics forward. The precise statement is in Lemma \ref{lem:measure_distortion_local_stable}.

The essence of the following proposition is that the conditional measures of $\hat{\mu}$ along stable and unstable leaves have a continuous Jacobian with respect to pushforward.
\begin{proposition}\label{prop:ac_of_pushforward}
\cite[Prop.~3.1]{butler2018measurable} There are continuous functions $J_u\colon \Sigma^u\to (0,\infty)$ and $J_s\colon \Sigma^s\to (0,\infty)$ such that for each $1\le i,j\le l$ and each Borel subset $K_u\subseteq [0;i]^u$,
\[
\hmu^u_i(K_u)=\int_{\hsigma_u^{-1}(K_u)\cap [0;j,i]^u} J_u\,d\hmu^u_j,
\]
and likewise for each Borel subset $K_s\subseteq [0;i]^s$.
\end{proposition}


In the following lemma, we consider the preimage of a local stable set $W^s_{loc}(\hx)$. This preimage may intersect several of the cylinders $[0;i]$. For convenience, we write $\hat{\mu}^s_{\hx,j}$ for the conditional measure on $\hsigma^{-1}(W^s_{loc}(\hx)\cap [0;j])$.

\begin{lemma}\label{lem:pushforward_measure_formula}
Suppose that $\hat{\Sigma}$ is a subshift of finite type as described above defined by a matrix $M$.
We have the following formula for the invariance of the measure under the shift. There exists a positive continuous function $J^u\colon \hat{\Sigma}\to \R^+$ that is constant on local stable leaves such that we have the following product formula for the measure. For each $i$,
\[
\hat{\mu}\vert_{[0;i]}=\int_{\hat{x}\in W^u_{loc}(\hy)} \sum_{j\to i} \hsigma_*(\hat{\mu}^s_{\hat{x},j})J^u(\hat{x})\,d\hat{\mu}^u_i,
\]
where the subscript in the sum indicates we are summing over all symbols $j$ for which $ji$ is subword permitted by $M$.
    \end{lemma}
    \begin{proof}
To begin, we have from Proposition \ref{prop:ac_of_pushforward} that for each $\hy\in [0;j,i]$ exists $J^u_{ji}$ such that 
\begin{equation}\label{eqn:jacobian_for_mu_u}
\hsigma_*(\hat{\mu}^u_{\hy\vert_{[0;j,i]}})=J^u_{ji}d\hat{\mu}^u_{\hsigma(\hy)}.
\end{equation}
We now obtain a formula for the pushforward of $\hat{\mu}\vert_{[0;j,i]}$. This measure may be written as
\[
\int_{\hx\in W^u_{loc}(\hy)\cap \hat{\sigma}^{-1}[0;i]} \hat{\mu}^s_{\hx}\,d\hat{\mu}^u_{\hy}\vert_{W^u_{loc}(\hy)\cap [0;j,i]}.
\]
By pushing this forward by $\hat{\sigma}$, we obtain an expression for the measure on $[-1;j,i]\subseteq [0;i]$:
\[
\hat{\mu}\vert_{[-1;j,i]}=\int_{\hx\in W^u_{loc}(\hat{\sigma}(\hy))} \hat{\sigma}_*(\hat{\mu}^s_{\hat{\sigma}^{-1}(\hx)})\,d\hat{\sigma}_*(\hat{\mu}^u_{\hy}\vert_{W^u_{loc}(\hy)\cap [0;j,i]}).
\]
But from equation \eqref{eqn:jacobian_for_mu_u}, this is the same as 
\[
\int_{\hx\in W^u_{loc}(\hy))} \hat{\sigma}_*(\hat{\mu}^s_{\hat{\sigma}^{-1}(\hx)})J^u_{ji}(\hx)\,d\hat{\mu}^u_{\hat{\sigma}(\hy)}.
\]
Thus summing over all symbols $j$ that may precede $i$, we obtain the expression promised by the lemma.    
    \end{proof}

The following fact is then straightforward to derive.

\begin{lemma}\label{lem:measure_distortion_local_stable}
Suppose that $\hat{\Sigma}$ is a shift and that $\mu$ is a $\hat{\sigma}$-invariant probability measure on $\hat{\Sigma}$ with continuous product structure. Then the conditional measures satisfy the following property: 
If $\hx\in \hat{\Sigma}$ and $K$ a subset of $W^s_{loc}(\hx)$ then for all $n\ge 0$ we have that 
\[
\frac{\hmu^s_{\hsigma^n(\hx)}(\hsigma^n(K))}{\hmu^s_{\hsigma^n(\hx)}(\hsigma^n(W^s_{loc}(\hx)))}=\frac{\hmu^s_{\hx}(K)}{\hmu^s_{\hx}(W^s_{loc}(\hx))}.
\]
\end{lemma}
\begin{proof}
Consider $\hsigma^{-1}(W^s_{loc}(\hx))$. Then there are points $\hy_1,\ldots,\hy_k$ such that 
\[
\hsigma^{-1}(W^s_{loc}(\hx))=\bigsqcup_{1\le i\le k} W^s_{loc}(\hy_i).
\]
From Lemma \ref{lem:pushforward_measure_formula}, because $J^u$ is constant along local stable leaves, it follows that there exist positive numbers $\xi_i$, $1\le i\le k$ such that
\[
\hmu^s_{\hx}=\sum_{1\le i\le k} \xi_i\hsigma_*(\hmu^s_{\hy_i}).
\]
By iterating on this formula we obtain the lemma.
\end{proof}

\subsection{Grassmannians, Arrangements, and Quasi-Projective Maps} Here we introduce the definitions related to Grassmannian cocycles necessary to study Lyapunov exponents with multiplicity 1. Throughout $\Lambda^p (\R^d)$ denotes the space of alternating $p$-forms over $\R^d$ and $\Gr(p,d)$ the Grassmannian of $p$-dimensional subspaces, which includes into $\P \Lambda^p(\R^d)$ as the decomposable $p$-forms, i.e., the $p$-forms $\eta$ which may be written as $\eta = v_1 \wedge \dots \wedge v_p$.  In what follows we endow $\Gr(k,d)$ with the metric:
$$d_{\Gr(k,d)}(V,W) = \sup_{v \in V, ||v|| = 1} \inf \{\|v-w\|: w \in W\}.$$

Any $\omega\in \Lambda^{d-p}(\R^d)$ defines a \emph{hyperplane} $H_{\omega} \subseteq \Lambda^p(\R^d)$ which is defined to be the kernel of the map $\eta\mapsto \eta\wedge \omega$. When $\omega$ is defined by a decomposable $(d-p)$-form, we think of $\omega$ as being ``geometric" because it defines a $p$-dimensional subspace of $\R^d$. Hence we may also define \emph{geometric} hyperplanes in $\Lambda^p(\R^d)$. For a $d-p$-form $\omega$, the intersection of $H_{\omega}$ with $\Gr(p,d)$ is called a \emph{hyperplane section} of $\Gr(p,d)$. Hence when a hyperplane section of $\Gr(p,d)$ is defined by a decomposable $(d-p)$-form $\omega$, we call it a \emph{geometric} hyperplane section. We call a finite intersection of hyperplane sections a \emph{linear section} of $\Gr(p,d)$. In the case that a linear section of $\Gr(p,d)$ is defined by geometric hyperplane sections, we call it a \emph{geometric} linear section. Finally, a finite union of linear sections is called a \emph{linear arrangement}, and when we take a finite union of geometric linear sections, we call this a \emph{geometric} linear arrangement. Similarly, one may define a \emph{geometric subspace} of $\Lambda^p(\R^d)$ as a finite intersection of geometric hyperplanes in $\Lambda^p(\R^d)$.


Observe that in $\Gr(1,d)$ linear sections are just subspaces of $\R^d$. Linear arrangements play an analogous role in defining irreducibility conditions as the classic strong irreducibility condition does in Furstenberg's theorem on positivity of the top Lyapunov exponent.

The following lemma shows that linear arrangements indeed carry some properties of finite unions of linear subspaces to higher Grassmannians:

\begin{lemma} \label{linearsecsproperties} \cite[Lemma 8.4]{viana2014lectures} If $\{\cL_\alpha: \alpha \in I\}$ is an arbitrary family of linear arrangements in $\Gr(p,d)$ then $\cap_\alpha \cL_\alpha$ coincides with the intersection of the $\cL_\alpha$ over a finite family. Hence the family of linear arrangements is closed under finite union and arbitrary intersections and moreover for $B \in \GL(d)$ the inclusion $B(\cL) \subseteq (\cL)$ implies that $B(\cL) = \cL$.
\end{lemma}

The space of linear sections $\cS = \P S \cap \Gr(p,d)$, where $S$ denotes a subspace of $\Lambda^p(\R^d)$, is topologized via the topology for linear subspaces of $\Lambda^p(\R^d)$. Namely, we say that $\cS_n \to \cS$ if and only if $S_n \to S$. 

In \cite[Chapter 8]{viana2014lectures}, it is claimed that the space of geometric linear sections is compact in order to prove simplicity of spectrum under a version of twisting and pinching. However, as we see in Example \ref{ex:counterexample} below, there seems to be an oversight in this claim. Because of this we state irreducibility criteria in terms of linear sections that are not necessarily geometric.

The dimension of a linear section $\cS = \P S \cap \Gr(p,d)$ is defined simply as the dimension of $S$ as a subspace of $\Lambda^p (\R^d)$. Once we remove the geometric restriction, we regain compactness of the space of linear sections of a fixed dimension:

\begin{example}\label{ex:counterexample} In this example, we show that if $D_n$ is a sequence of geometric linear sections of a Grassmannian, then $\lim_n D_n$ is not necessarily a geometric linear section: it is still a linear section, but it may no longer be geometric. We construct our example in $\Gr(2,4)$.

Let $\eta = e_1 \wedge e_2 \in \Lambda^2 (\R^4)$ and $\omega_n = (e_1 + n^{-1}e_3) \wedge (e_2 + n^{-1}e_4) \in \Lambda^2 (\R^4)$. We will write $e_{ij} := e_i \wedge e_j$. Then $H_\eta$ and $H_{\omega_n}$, the hyperplane sections defined by $\eta$ and $\omega_n$ respectively, are:
$$\begin{aligned}
    H_\eta &= \text{span} \langle e_{12}, \,e_{13}, \,e_{14}, \,e_{23}, \, e_{24}\rangle,  \\
    H_{\omega_n} &= \text{span} \langle e_{13}, \, e_{24},\, e_{12}-n^{-2}e_{34},\, e_{14} + e_{23}, \, e_{14} + n^{-1}e_{34} \rangle.  \\
\end{aligned}
$$
Thus
$$S_n := H_\eta \cap H_{\omega_n} = \text{span} \langle e_{13}, \, e_{24}, \, e_{14}+e_{23}, \, e_{12}+n^{-1}e_{14} \rangle,$$
is by definition a sequence of geometric linear sections when intersected with $\Gr(2,4)$. Call $D_n=S_n\cap \Gr(2,4)$, which may be identified with the decomposable tensors $S_n\subseteq \Lambda^2(\R^4)$.  We claim that $D=\lim_{n} D_n$ is not a geometric linear section of $\Gr(2,4)$. 

To see this we will first find a description of what $D$ is, then we will show that $D$ cannot be geometric because any decomposable $1$-form vanishing when wedged with all elements of $D$ is a multiple of $e_{12}$.

To proceed we use that $\omega\in \Lambda^2(\R^4)$ is decomposable if and only if $\omega^2=0$. By using this we may calculate which tensors in $S_n$ are decomposable. Namely we see that 
\[
(a(e_{12}+n^{-1}e_{14})+be_{13}+ce_{24}+d(e_{14}+e_{23}))^2=0
\]
if and only if $d^2+adn^{-1}-bc=0$. 

If we let $S=\lim_n S_n$ then 
\[S = \text{span} \langle e_{12}, \, e_{13}, \, e_{24}, \, e_{14}+e_{23} \rangle,
\]
In particular, if we calculate as before, we see that $S\cap \Gr(2,4)$ is equal to the tensors of the form
\[
(ae_{12}+be_{13}+ce_{24}+d(e_{14}+e_{23}))^2
\]
with $d^2-bc=0$. Thus by comparing pointwise limits for different values of $a,b,c,d$, we find that $D=S\cap \Gr(2,4)$. In particular, note that some of the forms in $D$ are $e_{12},e_{13},e_{24}$ and $(e_1+e_2)\wedge (e_3+e_4)$. Thus any $\omega\in \Lambda^2(\R^4)$ that vanishes when wedged with things in $D$ must be of the form
\[
ae_{12}+be_{14}+ce_{23}
\]
which is decomposable if and only if $bc=0$. In the cases that $a\neq 0$ and $b=0$, the form we get is a multiple of $e_2\wedge(e_1+ce_3)$, which does not vanish when wedged with $(e_1+e_2)\wedge (e_3+e_4)$ unless $c=0$. So, the form must be $e_{12}$. Similarly, if $c=0$, the form is a multiple of $e_1\wedge(e_2+be_4)$ and we must have $b=0$ as before. Thus if $a\neq 0$, the only forms we get are multiples of $e_{12}$. If $a=0$, then we get $e_{14}$ and $e_{23}$ neither of which vanish when wedged with $(e_1+e_2)\wedge (e_3+e_4)$. Thus the only decomposable forms vanishing when wedged with everything in $D$ are multiples of $e_{12}$. 

But this means that $D$ cannot be a geometric linear section of $\Gr(2,4)$ because it is not defined by $H_{e_{12}}$. Therefore the space of geometric linear sections is not compact.
\end{example}

When passing to the limit of the actions of linear maps on Grassmannians, it is useful to work with quasi-projective maps. We recall some definitions below. 

If $P\in \GL(d,\R)$ is a linear map then it induces a map $P_{\#}\colon \Gr(k,d)\to \Gr(k,d)$ for each $1\le k \le d$. We call such maps projective and often we denote the action of $P$ on Grassmannians simply by $P$ when there is no ambiguity. 

The space of projective maps on $\Gr(k,d)$ admits a very useful natural compactification which we now describe. Recall that we may embed $\Gr(k,d)$ in $\P \Lambda^k(\R^d)$, so the space of projective maps of $\Gr(k,d)$ embeds into the space of linear maps of $\Lambda^k(\R^d)$ with norm 1 with respect to the natural norm of $\Lambda^k(\R^d)$. Since the latter is clearly compact, the closure of the set of projective maps regarded as above must also be, and its elements are referred to as \textit{quasi-projective maps}. For more details of this construction, we refer the reader to \cite[Section 2.1]{avila2007simplicity}.

A quasi-projective map $Q$ has a well-defined kernel when regarded as a linear map on $\Lambda^k(\R^d)$. We define $\ker Q \subseteq \Gr(k,d)$ to be the intersection of its kernel as a linear map on $\Lambda^k(\R^d)$ with $\Gr(k,d)$, again using the embedding of $\Gr(k,d)$ in $\P \Lambda^k(\R^d)$. Since linear maps of the form $\Lambda^k P$ map rank 1 tensors to rank 1 tensors in $\Lambda^k(\R^d)$ and the space of rank 1 tensors is compact in $\P\Lambda^k(\R^d)$, for $\xi \in \Gr(k,d) \setminus \ker Q$, we may define $Q(\xi) \in \Gr(k,d)$ using the corresponding action of $Q$ on $\P \Lambda^k(\R^d)$. Hence quasi-projective maps can be regarded as functions $Q: \Gr(k,d) \setminus \ker Q \to \Gr(k,d)$.

\begin{lemma} \cite[Lemma 2.3]{avila2007simplicity} The kernel of any quasi-projective map $Q$ on $\Gr(k,d)$ is contained in some hyperplane section of $\Gr(k,d)$.
\end{lemma}

For a measure $\nu$ on $\Gr(k,d)$, the pushforward measure $Q_* \nu$ is well-defined as long as $\nu(\ker Q) = 0$. The following lemma will be very useful in our proof:
\begin{lemma}
\cite[Lem.~2.4]{avila2007simplicity} \label{lem: quasiprojectiveconvergence }If $(P_n)$ is a sequence of projective maps converging to some
quasi-projective map $Q$ of $\Gr(k,d)$ and $(\nu_n)$ is a sequence of probability
measures in $\Gr(k,d)$ converging weakly to some probability $\nu$ with $\nu(\ker Q)=0$
then $(P_n)_*\nu_n$ converges weakly to $Q_*\nu$.
\end{lemma}

\subsection{Invariant \emph{u}-states} The cocycle $\hA$ induces a cocycle on the Grassmannian bundle $\hat{\Sigma}\times \Gr(k,d)$ in the natural way, which by abuse of notation we also denote by $\hA$. Similarly, the stable and unstable holonomies induce maps on the Grassmannians which we denote by the same symbols.

\begin{definition}\label{def:ustate} For a fiber bunched cocycle $\hA\colon \hSigma \to \GL(d,\R)$ and a shift-invariant measure $\hmu$ on $\hSigma$, we say that a probability measure $\hm$ on $\hSigma \times \Gr(k,d)$ is a \emph{$u$-state} over $\hat{\mu}$ if it is invariant by the cocycle and 
$$H^u_{\hx \hy} \hm_{\hx} = \hm_{\hy},$$
for any $\hx$ and $\hy$ in the same local unstable leaf, and where $\{\hm_{\hx} : \hx \in \hSigma\}$ is a family of measures on $\Gr(k,d)$ given by the disintegration of $\hm$ along the partition $\{\{\hx\} \times \Gr(k,d): \hx \in \hSigma\}$. 
\end{definition}

When $\hmu$ has continuous local product structure, $u$-states have several other useful properties. We assume that $\hA$ is a fiber bunched cocycle and $\hmu$ is a measure with continuous local product structure. The following propositions from \cite{avila2007simplicity} sum up the relevant properties about $u$-states used in our proof:

\begin{proposition} \cite[Proposition 4.2]{avila2007simplicity} There exists a $u$-state $\hm$ on $\Gr(k,d)$.
\end{proposition}

 The projection $m$ of an invariant measure $\hm$  on $\hSigma \times \Gr(k,d)$ to $\Sigma \times \Gr(k,d)$ similarly admits a disintegration over the fibers $\{x\} \times \Gr(k,d)$, which we denote by $\{m_x : x \in \Sigma\}$. A martingale convergence argument relates $\hm_{\hx}$ to $m_x$ for any invariant measure (not necessarily a $u$-state):
 
 \begin{proposition} \label{prop: martingale} \cite[Corollary 3.4]{avila2007simplicity} Let $\hm$ be an invariant probability measure on $\hSigma \times \Gr(k,d)$ which projects to $\hmu$. For $\hmu$-a.e. $\hx \in \hSigma$:
 $$\hm_{\hx} = \lim_{n \to \infty} A^n(\hsigma^{-n}(\hx))_* m_{P(\hsigma^{-n}(\hx))},$$
 where the limit is taken in the weak* topology, and $P$ is the projection onto the one-sided shift $\Sigma^u$ as before.
 \end{proposition}

 Moreover, the following also holds for any $m$, not necessarily a $u$-state:

\begin{lemma}\label{lem:cts_disintegration_ustate} \cite[Lemma 4.5]{avila2007simplicity} Let $\hm$ be an invariant probability measure on $\hSigma \times \Gr(k,d)$ which projects to $\hmu$. Then the measures $\{m_x: x \in \Sigma\}$ given by 
$$m_x = \int_{W^s_{loc}(x)} \hm_{\hx} \, d\hmu^s_x(\hx), $$
define a disintegration of $m$ along $\{\{x\} \times \Gr(k,d): x \in \Sigma\}.$
\end{lemma}

When $m$ is a $u$-state, the lemma above and invariance by unstable holonomies additionally gives continuity of the disintegration $\{m_x: x \in \Sigma\}$:

\begin{proposition} \cite[Proposition 4.4]{avila2007simplicity} \label{prop: weakcont} Any $u$-state $m$ on $\Sigma \times \Gr(k,d)$ admits some disintegration $\{m_x : x \in \Sigma\}$ into measures on $\Gr(k,d)$ which vary weak* continuously with respect to $x$.
\end{proposition}

Lastly, using the continuity property above, one also obtains the following invariance statement for the disintegration $m_x$:
\begin{proposition} \cite[Corollary 4.7]{avila2007simplicity} \label{prop: invustates} If $m$ is a $u$-state and $m_x$ is the continuous disintegration of $m$ from Lemma \ref{lem:cts_disintegration_ustate}, then:
$$m_x = \sum_{z \in \sigma^{-k}(x)} \frac{1}{J\sigma^k(z)} A^k(z) m_z,$$
for every $x \in \Sigma$ and $k \geq 1$, and $J\sigma^k$ is a bounded measurable function for each $k \geq 1$.
\end{proposition}

\section{Strong Irreducibility and Distribution of Osceledec subspaces}\label{sec:smoothness}

The first key step in the proof of Furstenberg's theorem and in the main theorem in \cite{avila2007simplicity} is using an irreducibility hypothesis to prove that the measure of $u$-states disintegrated over the $1$-sided shift cannot be concentrated on subspaces. We prove the corresponding result in our setting in Proposition \ref{smoothconditionals}. For the rest of the paper we fix a fiber-bunched cocycle $\hat{A}:\hat{\Sigma} \to \GL(d, \R)$.  

\begin{definition} A function $\cL$ from a topological space $X$ to the set of linear arrangments in $\Gr(k,d)$ is said to be a continuous family of $k$-linear arrangements on $X$ if there is an $n \in \N$ such that for each $x \in X$
$$\cL(x) = \cS_1(x) \cup \dots \cup \cS_n(x),$$
where each $\cS_i(x)$ is of the form $\P S_i(x) \cap \Gr(k,d)$ and the $S_i$ are continuous functions from $X$ to the space of linear subspaces of $\Lambda^k(\R^d)$ of a fixed dimension $d(i)$. 
The continuous family of linear arrangements is said to be non-trivial if $\cL(x)$ is neither all of $\Gr(k,d)$ nor the empty set for all $x \in X$.
\end{definition}

\begin{definition} \label{def: fullirreducible} A continuous cocycle $\hat{A}\colon \hat{\Sigma} \to \GL(d, \R)$ is  strongly reducible on $\Gr(k,d)$ if there exists a non-trivial continuous family of $k$-linear arrangements $\cL$ on $\hat{\Sigma}$ invariant by $\hA$, i.e., for every $\hx \in \hSigma$:
$$\begin{aligned}
&\hA(\hx)_* \cL(\hx) = \cL(\hsigma(\hx)). \\
\end{aligned}
$$
The cocycle $\hA$ is strongly irreducible on $\Gr(k,d)$ if it is not strongly reducible on $\Gr(k,d)$. The cocycle is said to be fully irreducible if it is strongly irreducible on $\Gr(k,d)$ for all $1 \leq k \leq d-1$.
\end{definition}

Fix a shift invariant measure $\hat{\mu}$  on $\hat{\Sigma}$ with continuous local product structure. The following key proposition shows that $u$-states of fully irreducible fiber-bunched cocycles on Grassmannians must be smooth in a certain sense: 

\begin{proposition} \label{smoothconditionals} Suppose $\hA$ is strongly irreducible on $\Gr(k,d)$. Let $\hat{m}$ be a $u$-state on $\hat{\Sigma} \times \Gr(k,d)$ associated to $\hat{\mu}$.  For $x \in \Sigma$, let $\{m_x\}_{x\in \Sigma}$ be the disintegration of $m$ over the fibers of $\Sigma \times \Gr(k,d)$. Then $m_x(\cS) = 0$ for every $x \in \Sigma$ and any linear section $\cS$ of $\Gr(k,d)$. In particular $m_x(H) = 0$ for every $x \in \Sigma$ and any hyperplane section $H$ of $\Gr(k,d)$.
\end{proposition}

To begin the proof we prove that we can reduce the problem to the reduced cocycle from Equation \eqref{eqn:def_reduced_cocycle} and the one-sided shift. Note that because the reduced cocycle is constant on local stable leaves that it defines a cocycle over the one-sided shift $\Sigma$.

\begin{proposition}
For every $1 \leq k \leq d$, there does not exist a non-trivial continuous family $\cL$ of $k$-linear arrangements on $\Sigma$ invariant by the cocycle on the one sided shift induced by the reduced cocycle from Equation \eqref{eqn:def_reduced_cocycle}.
\end{proposition} 
\begin{proof} Indeed, if there was such a $\cL$, then using the definition of $\phi^u$ appearing before Equation \eqref{eqn:def_reduced_cocycle} we could define $\hat{\cL}(\hx)$ on $\hat{\Sigma}$ by $ \hat{\cL}(\hx) = H^s_{\phi^u(\hx), \hx} \cL(x)$ so that
$$ \hat{A}(\hx)_* \hat{\cL}(\hx)  = H^s_{\phi^u(\hsigma \hx), \hsigma \hx} A(\hx)_* \cL(x) = H^s_{\phi^u(\hsigma \hx), \hsigma \hx} \cL(\sigma x) = \hat{\cL}(\hsigma \hx),$$
which yields  invariance of $\hat{\cL}$. Moreover, since stable holonomies are continuous in the leaf and transverse direction, continuity of $\hat{\cL}$ follows from continuity of $\cL$.
\end{proof}

As $\hat{A}$ is by definition conjugate to $A$, for the goal of proving simplicity of spectrum it suffices to work with $A$. Hence, from now on we work only with the reduced cocycle $A$.

\begin{proof}[Proof of Proposition \ref{smoothconditionals}] 

We take the continuous disintegration of $m$ from Equation \eqref{lem:cts_disintegration_ustate}. Suppose that $m_x(\cS) > 0$ for some $x \in \Sigma$ and some linear section $\cS$ of $\Gr(k,d)$. Let $d_0$ be the minimal dimension such that there exists a subspace $S$ of $\Lambda^k(\R^d)$ such that $m_x(\P S \cap \Gr(k,d)) > 0 $ for some $x \in \Sigma$. Let
\begin{equation}\label{eqn:def_gamma_0}
\gamma_0=\sup_{x\in \Sigma, \mathbb{P}S: \dim S=d_0}m_x(\P S \cap \Gr(k,d)).
\end{equation}

We claim that by compactness $\gamma_0$ is realized. The space of all $\mathbb{P}S$ is compact and $m_x$ varies continuously in $x$ in the weak* topology by Proposition \ref{prop: weakcont}. Hence if we take any sequence $x_i, S_i$ such that 
\[
m_{x_i}(\mathbb{P}S_i\cap \Gr(k,d))\to \gamma_0,
\]
we may pass to a subsequence so that we may assume that $x_i\to x'$ and $S_i\to S$. Consider the set $B_{\epsilon}(S)$, then weak* convergence implies that for each $\epsilon$ as $B_{\epsilon}(S)$ is closed
\[\begin{aligned}
m_x(\mathbb{P}B_{\epsilon}(S)\cap \Gr(k,d))&\ge \limsup_{i \to \infty} m_{{x_i}}(\mathbb{P}B_{\epsilon}(S_i)\cap \Gr(k,d)) \\
&\ge \limsup_{i \to \infty} m_{x_i}(\mathbb{P}S_i\cap \Gr(k,d))=\gamma_0.
\end{aligned}\]
Letting $\epsilon\to 0$, we obtain that $m_{x'}(\mathbb{P}S\cap \Gr(k,d))=\gamma_0$ by continuity from above of the measures.

We now claim that for each $x \in \Sigma$, there exists at least one subspace $S(x)$ of dimension $d_0$ such that $m_x(S(x) \cap \Gr(k,d)) = \gamma_0$. The proof is along the same lines as in \cite[Lemma 5.2]{avila2007simplicity}. First, one shows that over every cylinder $[J]$ that the supremum of $m_x(S\cap \Gr(k,d))$ is $\gamma_0$. Then one can pick a sequence converging to any point $x'\in \Sigma$ as above. We are able to pass to this limit because the space of linear sections is compact.

We now show that the linear sections realizing the supremum form a continuous family. First, we show that there exists $\ell$ such that at every point $x\in \Sigma$ there exist exactly $\ell$ distinct linear $d_0$-dimensional subspaces of $\Lambda^k(\R^d)$, $S_1(x),\ldots, S_{\ell}(x)$ such that $m_x(\P S_i(x) \cap \Gr(k,d))=\gamma_0$. Take a point $x\in \Sigma$ that maximizes the number $\ell$. Then by invariance of $m_x$, i.e. Proposition \ref{prop: invustates}:
$$m_x = \sum_{z \in \sigma^{-k}x} \frac{1}{J\sigma^k(z)} (A^k(z))^{-1}_* m_z, \text{ with } \sum_{z \in \sigma^{-k}x} \frac{1}{J\sigma^k(z)} = 1,$$
it follows that for each $z\in \sigma^{-k}(x)$ that 
$$m_{z}((A^k(z)^{-1})_*\P S_i(x) \cap \Gr(k,d)) = \gamma_0, \text{ for all }i =  1, ..., n.$$ 

Note that the preimages of $x$ become dense in $\Sigma$. Thus for any point $q\in \Sigma$, we may take a sequence of $z_n$ such that $z_n\to q$ and $m_{z_n}$ gives measure $\gamma_0$ to exactly $\ell$ subspaces $S_1(z_n),\ldots, S_\ell(z_n)$. By the argument mentioned after Equation \eqref{eqn:def_gamma_0}, we may pass to a subsequence so that for each $i$ $S_i(z_n)\to S_i'$, where $S_i'$ is some $d_0$-dimensional subspaces of $\Lambda^k(\R^d)$. We claim that all the $S_i'$ are distinct. But this is clear: if two coincided, then the preceding argument implies that the limiting measure of some $S_i'$ would be at least $2\gamma_0$, which is impossible as $\gamma_0$ is the maximum.

Now we show that these sections vary continuously as $x$ varies. First observe that if $x_n\to x$ then for each $n$ we must have that $S_i(x_n)$ is near some $S_j(x)$---otherwise by taking a limit as before we would obtain an additional section $S'$ of measure $\gamma_0$ for $m_x$.


Finally, to see that the family $\{S_1(x),\ldots, S_\ell(x)\}_{x\in \Sigma}$ is invariant under the cocycle, note that this holds on the set of pre-images of $z$ and that this set is dense. Thus invariance holds by continuity. But we have now reached a contradiction of our assumption of irreducibility on $\Gr(k,d)$, so we are done.
\end{proof}


\subsection{Distribution of Osceledec subspaces}\label{subsec:distribution_of_osceledec_subspaces}


If we know a priori that for $\hmu$ we have $\lambda_{k}(\hmu) > \lambda_{k+1}(\hmu)$ then $u$-states over $\hat{\mu}$ are supported on $E^u_k$. The following proposition makes this precise. 
\begin{proposition} \label{prop: ustateslieondirac}
   Suppose $\hmu$ has full support and continuous local product structure and $\lambda_{k}(\hmu) > \lambda_{k+1}(\hmu)$. Let $\hm$ be any $u$-state of $\hmu$ on $\hSigma \times \Gr(k, d)$. Then for any $x \in \Sigma$ and $\hmu^s_x$-a.e. $\hy \in W^s_{loc}(x)$:
   $$\hm_{\hy} = \delta_{E^u_k(\hy)}.$$
\end{proposition}
\begin{proof}
    Recall that by Proposition \ref{prop: martingale}, for $\hmu$-a.e.\ $\hx \in \hSigma:$
    $$\hm_{\hx} = \lim_{n \to \infty} A^n(\hsigma^{-n}(\hx))_* m_{P(\hsigma^{-n}(\hx))},$$
    where $P\colon \hSigma \to \Sigma$ is the projection from the two-sided to the one-sided shift.

    Passing to a subsequence $n_j$ such that $\hsigma^{-n_j}(\hx) \to \hx_0$, since $x \mapsto m_x$ is weak$*$ continuous:
    $$\lim_{j \to \infty}m_{P(\hsigma^{-n_j}(\hx))}= m_{x_0},$$
    where, as usual, $x_0 = P(\hx_0)$. Moreover, by passing to a further subsequence if needed, since the space of quasi-projective maps is compact we may assume that $A^{n_j}(\hsigma^{-n_j}(\hx))$ converges to some quasi-projective map $Q_0$. By Oseledec' theorem since $\lambda_{k}(\hmu) > \lambda_{k+1}(\hmu)$, the sequence of maps $A^n(\hsigma^{-n}(\hx))$ has $$
    \frac{\sigma_k(A^n(\hsigma^{-n}(\hx)))}{\sigma_{k+1}(A^n(\hsigma^{-n}(\hx)))} \to \infty$$ for $\hmu$-a.e.\ $\hx$, so the the image of $Q_0$ in $\Gr(k,d)$ must be the single point in $\Gr(k,d)$ given by 
    $$\lim_{n\to \infty} U_k(\hsigma^{-n}(\hx), n) = E^u_k(\hx),$$
    again by Oseledec' theorem. Finally, since by Proposition \ref{smoothconditionals} $m_{x_0}(\ker Q_0) = 0$, we apply Lemma \ref{lem: quasiprojectiveconvergence } to conclude that
    $$\hm_{\hx} =\lim_{j \to \infty} A^n(\hsigma^{-n_j}(\hx))_* m_{P(\hsigma^{-n_j}(\hx))} = {Q_0}_* m_{x_0} = \delta_{E^u_k(\hx)}$$
    holds for $\hmu$-a.e.\ $\hx$. In particular, by the disintegration formula, for $\mu$-a.e.\ $x \in \Sigma$ the statement must hold for $\hmu^s_x$-a.e.\ $\hy \in  W^s_{loc}(x)$.

    Now we want to upgrade this statement to \textit{every} $x \in \Sigma$. Fix some $x_0\in \Sigma$ for which the statement holds for $\hmu^s_{x_0}$-a.e.\ $\hy \in  W^s_{loc}(x_0)$ and fix an arbitrary $x \in \Sigma$ which lies in the same $[0:i]$ cylinder. Because $\hat{m}$ is a $u$-state, for $\hy \in  W^s_{loc}(x_0)$ and $\hz = W^u_{loc}(\hy) \cap W^s_{loc}(x)$, 
    \[
    H^u_{\hy \hz}\hm_{\hy} = \hm_{\hz},
    \]
    so that $\hm_{\hz} = \delta_{H^uE^u_k(\hy)}$ for $\hmu^s_{x_0}$-a.e.\ $\hy \in  W^s_{loc}(x_0)$, since the measures $\hmu^s_{x_0}$ and $\hmu^s_{x}$ are equivalent due to the continuous product structure. But since $H^u_{\hy \hz} E^u_k(\hy) = E^u(\hz)$ by Lemma \ref{lem:unstable_bundles_cts_along_unstable_leaves} and there are only finitely many cylinders $[0:i]$, the proof is complete. 
\end{proof}

\begin{remark} It is natural to consider if the correct definition of strong irreducibility for cocycles admitting invariant holonomies should be weakened by requiring non-existence of a continuous family $\cL$ which is in addition invariant by the holonomies. However, in \cite[Remark 3.11]{stoyanovgouezel} the authors show that with this weaker definition, which they call holonomy-irreducibility, the conditional measures of $u$-states do not have to lie on a single Dirac mass, so  Proposition \ref{prop: ustateslieondirac} does not hold. Their example suggests that in the study of cocycles admitting canonical holonomies it is more natural to not include invariance by holonomies in the definition of irreducibility.
\end{remark}

\begin{remark} \label{rem: unreduced} Since the conjugacy between the unreduced and reduced cocycles preserves the unstable directions and conjugates the stable holonomies to the identity, if we consider instead the unreduced cocycle $\hA$, the statement of Proposition \ref{smoothconditionals} may still be interpreted as saying that given any hyperplane section $H$ of $\Gr(k,d)$, any $\hx \in \hSigma$ then for $\hmu^s_x$-a.e.~point $\hy$ in $W^s_{loc}(\hx)$ the conditional measure $\hm_{\hy}$ gives $0$ measure to $H^u_{\hx \hy}H$. In particular if $\hmu$ has $\lambda_k(\hmu) > \lambda_{k+1}(\hmu)$, we have $\hm_{\hy} = \delta_{E^u_k(\hy)}$ for $\hmu$-a.e. $\hy$ by the preceeding proposition, so that for every $\hx \in \hSigma$ and $\hmu^s_{\hx}$-a.e.~point $\hy$ in $W^s_{loc}(\hx)$ one must have $E^u_k(\hy) \notin H^s_{\hx \hy}H$ for the unreduced cocycle. 

Similarly, the analogous statement holds for the stable directions along an unstable manifold, that is, for a fixed hyperplane section $H \subseteq \Gr(d-k, d)$ and a fixed $\hx \in \hSigma$, for $\hmu^u_{\hx}$-a.e. point $\hy$ in $W^u_{loc}(\hx)$ it holds that $E^s_{d-k}(\hy) \notin H^u_{\hx \hy}H$. 
\end{remark}

\section{Convergence to a Dirac measure} \label{sec: dirac}

Now we turn to the proof of Theorem \ref{theorem: maintheorem}. 
Throughout, we fix $1 \leq k \leq d-1$ and $\hmu$ and $\heta$ as in the statement of Theorem \ref{theorem: maintheorem}. The key step to obtain $\lambda_k(\heta) > \lambda_{k+1}(\heta)$ is to prove that for $u$-states of $\hat{\eta}$ the conditional measures on $\Gr(k,d)$ are Dirac measures. 

\begin{proposition} \label{prop: convergencetodirac} Let $\hm^\eta$ be an invariant $u$-state for $\heta$ on $\hSigma \times \Gr(k, d)$. Then for $\heta$-a.e. $\hx$, the conditional measure $\hm^{\eta}_{\hx}$ is a Dirac delta on $\Gr(k, d)$. That is, there exists some $\xi^u(\hx) \in \Gr(k, d)$ such that
$$\hm^{\eta}_{\hx} = \delta_{\xi^u(\hx)}.$$
\end{proposition}

We give the proof of Proposition \ref{prop: convergencetodirac} at the end of this section. The existing approaches showing that cocycles have non-trivial Lyapunov exponents all rely on some form of ``twisting" and ``pinching". The approach we take to show the convergence to Dirac measures is similar to preexisting ones: we assume a form of pinching and twisting. Our pinching assumption is the existence some other ergodic fully supported measure $\hmu$ with continuous local product structure and $\lambda_{k}(\hmu) > \lambda_{k+1}(\hmu)$, as in Theorem \ref{theorem: maintheorem}. To construct a twisting setup in our setting we use a novel approach using both strong irreducibility (Section \ref{sec:smoothness}) combined with the behavior of generic orbits for $\hmu$. We then conclude the proof in Section \ref{sec: proof}.

To begin the construction we fix a periodic point in $\hSigma$ with an exponent gap which we will use to separate the exponents of generic orbits via shadowing. Since $\lambda_{k}(\hat{\mu}) >\lambda_{k+1}(\hmu)$ by the periodic approximation theorem \cite[Theorem 1.4]{kalinin2011livsic} there exists some periodic point $\hat{p} \in \hSigma$, such that the Lyapunov spectrum of $A$ over the orbit of $\hat{p}$ is arbitrarily close to the Lyapunov spectrum of $A$ with respect to $\hat{\mu}$. Hence we may take $\hp$ so that the $k$-th eigenvalue of $A^q(\hp)$, where $q$ is the period if $\hp$, is greater than the $(k+1)$-th eigenvalue. Without loss of generality we will assume that $\hp$ is a fixed point, i.e. $q=1$ and $\hp$ will be fixed for the rest of this section.

For the twisting construction, we will adopt the following notation. For $\hx \in \hSigma$ we define the unstable and stable cylinders $C^{u,s}_{l}(\hx)$ and the rectangle $R_l(\hat{x})$ of depth $l$ at $\hx$, which is the ball of radius $2^{-l}$ at $\hat{x}$, by:
$$\begin{aligned}
    &C^s_{l}(\hx) := \{\hx  = (\dots, y_{-1}, y_0, y_1, \dots)\in \hSigma: y_i = x_i, \,\forall \, 0 \leq i \leq l\},\\
    &C^u_{l}(\hx) := \{\hx  = (\dots, y_{-1}, y_0, y_1, \dots)\in \hSigma: y_i = x_i, \,\forall \, -l \leq i \leq 0\},\\
    &R_l(\hat{x}) := C^s_{l}(\hx) \cap C^u_{l}(\hx).
\end{aligned}$$

Observe that if $\hx \in [0:i]$, then $R_0(\hx) =[0:i]$ by definition. Since $k$ is fixed we will simply take $\Gr$, $E^u$ and $E^s$ to mean $\Gr(k, d)$, $E^u_{k}$, $E^s_{d-k}$, respectively. We will also use the following simple linear algebra lemma in the proof, where the angle $\angle (V,W)$ between two subspaces $V,W$ not necessarily of the same dimension is defined simply as the smallest angle between a vector in $V$ and a vector in $W$: 

\begin{lemma} \label{lem: conelemma}
    Fix $d \geq 1$. Given $\theta, \gamma > 0$ there is a $K> 1$ with the following property. Let $A \in \GL(d,\R)$ satisfy $\sigma_k(A) > K \sigma_{k+1}(A)$. Then for any $U$ of   dimension $k$ such that $\angle( U,S_{d-k}(A))> \theta$, one has:
    $$d_{\Gr}(A(U), U_k(A)) < \theta.$$
    Moreover, $A$ acts $\gamma$-Lipschitz in the set of $U$ of   such that $\angle( U,S_{d-k}(A))> \theta$ in the Grassmannian.
\end{lemma}
\begin{proof}
    By the singular value decomposition, this reduces to the analogous claim for diagonal matrices, which is straightforward because every unit vector $Av$ uniformly transverse to $S_{d-k}(A)$, $Av$ has $U_k(A)$ component uniformly bounded below.
\end{proof} 

In the following sequence of lemmas, we construct a sequence of rectangles  such that points in the rectangles shadow $\hp$ under $\hsigma^{-1}$ and such that under $\hsigma$ the cocycle over the points move the direction expanded by shadowing $\hp$ in the past away from fixed hyperplane sections, which will serve as ``twisting" orbits. We start with the construction of one of these rectangles in Lemma \ref{lem:good_rectangles} and conclude the construction in Proposition \ref{prop:twisting_sets}.

From now on for every point $\hx \in \hSigma$, we will use superscripts\footnote{Recall that subscripts already denote the $i$-th symbol of a word.} to denote preimages under $\hsigma$, e.g. $\hx^n = \hsigma^{-n}(\hx).$ Morever, observe in the following proofs that because we have taken the cocycle to be constant along local stable manifolds that  $A^n(\hat{y})=A^n(\hat{z})$ for $\hz \in W^s_{loc}(\hy)$.

\begin{lemma}\label{lem:good_rectangles}
Fix some sufficiently small $\theta_0>0$ and let $\hat{p}\in \hat{\Sigma}$. Suppose that $\hat{q}$ is Osceledec regular and generic for $\hat{\mu}$. Then for all $\epsilon>0$ and $l\in \N$, there exist $\rho$ and $k$ such that 
\begin{enumerate}
    \item
    $\hat{\sigma}^{-k}(R_{\rho}(\hat{q}))\subseteq C^u_l(\hat{p})$
    \item
    For any $\hat{y}\in \hat{\sigma}^{-k}(R_{\rho}(\hat{q}))$ which is Oseledec regular, let $\hat{z}=[\hat{y},\hat{p}]$. Then
    \[
    d_{Gr(k,d)}(A^k(\hat{y})E^u(\hat{z}),E^u(\hat{q}))<\epsilon.
    \]
    \item
    Further, for $\hy$ and $\hz$ as above, $A^k(\hat{y})\colon \Gr(k,d)\to \Gr(k,d)$ is $\gamma$-Lipschitz on a $\theta_0$-neighborhood of $E^u(\hat{z})$.
\end{enumerate}
\end{lemma}

\begin{proof}
We begin with the following claim.
\begin{claim} \label{claim: proof of rectangles}
    For every $\theta_0$ sufficiently small and any $l > 0$, there exists a measurable set $G_l\subseteq C^u_{2l}(\hat{p})$ with the following properties:
    \begin{enumerate}
        \item[(1')]
        $G_l$ has positive $\hmu$ measure.
        \item[(2')]
        For $\hat{y}\in G_l$ which is Oseledec regular,
        \[
        \angle (E^u(\hz), E^s(\hy))> 2 \theta_0, \text{ where }\hz =[\hat{y},\hat{p}].
        \]
        \item[(3')]
        $G_l$ is saturated by $W^s_{loc}$-leaves in $C^u_{2l}(\hat{p})$. That is, for $\hx \in G_l$ if $\hz \in W^s_{loc}(\hx) \cap C^u_{2l}(\hp)$, then $\hz \in G_l$. 
        
        \item[(4')]For $\hy \in G_l$ which is Oseledec regular the sequence of subspaces $(S(\hy, n))_n$ defined in Subsection \ref{ssec: lyapunov exponents} converges uniformly to $E^s(\hy)$.
    \end{enumerate}
\end{claim}
\begin{proof}
    Let $G_l$ be the set of points in $C^u_{2l}(\hp)$ which are Oseledec regular and satisfy condition (2'). For $y \in G_l$ let $\hz = [\hy, \hp]$.  Since $E^u(\hz) = H^u_{p, z}E^u(\hp)$, by Remark \ref{rem: unreduced} applied for $E^s$ and $H$ the hyperplane section defined by $E^u(\hp)$ on $\Gr(d-k,d)$, we see that for all $\theta_0$ sufficiently small $G_l \cap W^u_{loc}(\hp)$ has positive $\hmu^u_{\hp}$-measure. Observe that how much we need to shrink $\theta_0$ does not depend on $l$. Because $E^s$ is constant along $W^s_{loc}$-leaves, $G_l$ is saturated by $W^s_{loc}$-leaves in $C^u_{2l}(\hp)$, giving (3').  Moreover, since $\hmu$ has a local product structure $G_l$ is a set of positive $\hmu$-measure, so we verified (1'). 
    
    By restricting $G_l$ to a subset of positive $\hmu$-measure we may assume by Egorov's theorem that for $\hy \in G_l$ the sequence of subspaces $(S(\hy, n))_n$ converges uniformly to $E^s(\hy)$, verifying (4'). Observe that since the cocycle is constant along local stable leaves, passing to this subset of positive measure does not affect the saturation by $W^s_{loc}$-leaves in $C^u_{2l}(\hp)$, i.e., passing to this subset (3') still holds, so we are done. 
\end{proof}

Now we return to the proof of Lemma \ref{lem:good_rectangles}. Because $\hat{q}$ is generic for $\hmu$ and by (1') the set $G_l$ has positive measure there is a sequence of times $k_n \to \infty$ such that $\hq^{k_n} \in G_l$. Observe that for a given time $k_n$, we can always take $\rho$ to be sufficiently large so that (1) holds, since $G_l \subseteq C_{2l}^u(\hp)$. We will now determine how large $k_n$ has to be for (2) and (3) to hold, and, 
thus increasing $\rho$ accordingly if necessary at the end, we can do so ensuring that (1) holds as well.

First, we will show that for $k_n$ sufficiently large, statements of (2) and (3) hold for $\hy = \hq^{k_n}$ and $\hz = [\hq^{k_n}, \hp]$ exactly. In order to do this, we use the hyperbolicity of $\hq$ and Lemma \ref{lem: conelemma}. Then to conclude the proof we will later use continuity to show that (2) and (3) hold for any $\hy \in \hsigma^{-k_n}(R_\rho(\hq))$.

Applying Lemma \ref{lem: conelemma} with $\theta = \min\{\theta_0, \epsilon/2\}$ and $\gamma/2$ for the Lipschitz constant as in the statement of the lemma we may obtain a constant $K$ such that the conclusion of that lemma holds if $\sigma_k(A)>K\sigma_{k+1}(A)$ for a linear map $A\in \SL(d,\R)$. Since $\hq$ is generic for $\hmu$ and $\lambda_k(\hmu) > \lambda_{k+1} (\hmu)$, for the $K$ specified above if we choose $k_n$ large enough, then $\sigma_k(A^{k_n}(\hat{q}^{k_n}))>K\sigma_{k+1}(A^{k_n}(\hat{q}^{k_n}))$. Hence for any subspace $U$ with $\angle(U,S_{d-k}(\hat{q}^{k_n}, k_n)) > \theta_0$, we have that 
\begin{equation}\label{eqn:output_closeness}
d_{\Gr}(A^{k_n}(\hat{q}^{k_n})(U),U_k(\hat{q}^{k_n}, k_n))<\epsilon/2
\end{equation}
and $A^{k_n}$ is $\gamma/2$-Lipschitz on the same set of $U$. 

Since $\hq^{k_n} \in G_l$, by (2') in Claim \ref{claim: proof of rectangles} we have $\angle(E^u([\hq^k_n, \hp]), E^s(\hq^{k_n})) > 2 \theta_0$. Thus, since $S_{d-k}(\hat{q}^{k_n},k_n)$ converges to $E^s(\hat{q}^{n_k})$ by (4') in Claim \ref{claim: proof of rectangles}, for $k_n$ sufficiently large we see that we may take $U = E^u([\hq^{k_n}, \hp])$ in  (\ref{eqn:output_closeness}) to obtain that $A^{k_n}(\hq^{k_n})$ is $\gamma/2$-Lipschitz on a $\theta_0$-neighborhood of $E^u(\hz)$, verifying (3) for $\hy = \hq^{k_n}$. Moreover, since $U_k(\hat{q}^{k_n}, k_n)$ converges to $E^u(\hq)$, (\ref{eqn:output_closeness}) with $U = E^u(\hz)$ gives 
$$
d_{\Gr}(A^{k_n}(\hat{q}^{k_n})(E^u(\hz)),E^u(\hq))<\epsilon/2
$$
for $n$ sufficiently large. Thus, we have verified that (2) and (3) hold at exactly $\hy = \hq^{-k_n}$ as long as $k_n$ is sufficiently large and we now let $k := k_n$ for any such large $k_n$. 

Finally, we verify that the statements of (2) and (3) hold for all $\hy \in \sigma^{-k}(R_\rho(\hq))$ if $\rho$ is taken sufficiently large. The key observation is that $E^u(\hz)$, where $\hz = [\hy, \hp]$, varies continuously on $\hy$. Since the cocycle for a fixed $k$ also varies continuously, this implies that (2) and (3) are open conditions on $\hy$. Since (2) and (3) hold at $\hq^{k}$, increasing $\rho$ if necessary the same holds in the neighborhood $\sigma^{-k}(R_\rho(\hat{q}))$ of $\hq^{k}$, with $\epsilon$ in Equation \eqref{eqn:output_closeness} and $\gamma$-Lipschitzness as required for (3).
\end{proof}

    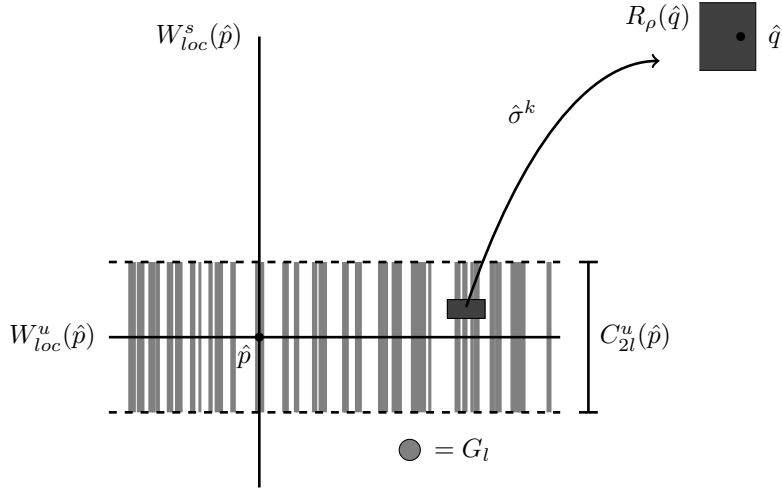
\begin{figure}
\begin{center}
\begin{tikzpicture}

\draw [line width = 2.435856134146136, color=gray] (2.0625146439797817,-1) -- (2.0625146439797817,1);
\draw [line width = 1.1197940714961911, color=gray] (-0.7900045879242279,-1) -- (-0.7900045879242279,1);
\draw [line width = 2.1022946817743984, color=gray] (-0.34915942613080553,-1) -- (-0.34915942613080553,1);
\draw [line width = 2.4909252014133205, color=gray] (3.178196188822154,-1) -- (3.178196188822154,1);
\draw [line width = 1.8858739656205856, color=gray] (3.8509864910646634,-1) -- (3.8509864910646634,1);
\draw [line width = 1.2419430563695284, color=gray] (2.719592981355696,-1) -- (2.719592981355696,1);
\draw [line width = 2.894448818084229, color=gray] (2.1658470934537863,-1) -- (2.1658470934537863,1);
\draw [line width = 2.3628560777590986, color=gray] (2.8871552703316823,-1) -- (2.8871552703316823,1);
\draw [line width = 1.922105103520559, color=gray] (2.096816933545668,-1) -- (2.096816933545668,1);
\draw [line width = 1.2195024543235418, color=gray] (-0.871257285039096,-1) -- (-0.871257285039096,1);
\draw [line width = 1.1072261541766089, color=gray] (3.381715921493708,-1) -- (3.381715921493708,1);
\draw [line width = 2.4912438139694006, color=gray] (0.3487089750989689,-1) -- (0.3487089750989689,1);
\draw [line width = 2.074321209189131, color=gray] (1.859275521928141,-1) -- (1.859275521928141,1);
\draw [line width = 2.872439926982289, color=gray] (3.4865306507033766,-1) -- (3.4865306507033766,1);
\draw [line width = 1.0232976727362866, color=gray] (3.357080250205982,-1) -- (3.357080250205982,1);
\draw [line width = 1.6262831639459001, color=gray] (-0.64818106108254,-1) -- (-0.64818106108254,1);
\draw [line width = 1.7259954041112269, color=gray] (-1.3518717987016942,-1) -- (-1.3518717987016942,1);
\draw [line width = 2.0018003397976223, color=gray] (3.0969758666702836,-1) -- (3.0969758666702836,1);
\draw [line width = 1.8394363496491488, color=gray] (2.047058184265434,-1) -- (2.047058184265434,1);
\draw [line width = 1.5403937365698372, color=gray] (-1.4132884763008484,-1) -- (-1.4132884763008484,1);
\draw [line width = 2.8660977193725956, color=gray] (3.485165477589738,-1) -- (3.485165477589738,1);
\draw [line width = 2.214916636312979, color=gray] (3.4215854721708983,-1) -- (3.4215854721708983,1);
\draw [line width = 2.386561784806678, color=gray] (-0.5518935989937539,-1) -- (-0.5518935989937539,1);
\draw [line width = 1.5179245053305113, color=gray] (-0.8956533787517784,-1) -- (-0.8956533787517784,1);
\draw [line width = 2.8830287043812577, color=gray] (-1.0703574038353176,-1) -- (-1.0703574038353176,1);
\draw [line width = 1.2063638085249653, color=gray] (3.1784620359028706,-1) -- (3.1784620359028706,1);
\draw [line width = 1.174409457526997, color=gray] (2.2663850283102693,-1) -- (2.2663850283102693,1);
\draw [line width = 1.581141694871307, color=gray] (0.4874116082076192,-1) -- (0.4874116082076192,1);
\draw [line width = 1.3952510212561282, color=gray] (0.5067475290892816,-1) -- (0.5067475290892816,1);
\draw [line width = 2.4603572906260833, color=gray] (1.8021611542046658,-1) -- (1.8021611542046658,1);
\draw [line width = 2.628921615299325, color=gray] (1.6238991588783485,-1) -- (1.6238991588783485,1);
\draw [line width = 1.647772437174985, color=gray] (-0.027168949176550994,-1) -- (-0.027168949176550994,1);
\draw [line width = 1.1916818322249954, color=gray] (2.826646106055012,-1) -- (2.826646106055012,1);
\draw [line width = 1.8413850487140124, color=gray] (1.1516740114833053,-1) -- (1.1516740114833053,1);
\draw [line width = 2.878471847330286, color=gray] (-1.6918008866946723,-1) -- (-1.6918008866946723,1);
\draw [line width = 1.3822621359641831, color=gray] (2.7437067406783004,-1) -- (2.7437067406783004,1);
\draw [line width = 2.625800542676249, color=gray] (1.3170605173532883,-1) -- (1.3170605173532883,1);
\draw [line width = 2.8949157280186135, color=gray] (-1.5778397002383946,-1) -- (-1.5778397002383946,1);
\draw [line width = 1.5829665542245717, color=gray] (-0.5100961674593014,-1) -- (-0.5100961674593014,1);
\draw [line width = 2.6514635349874016, color=gray] (-1.4289090682819077,-1) -- (-1.4289090682819077,1);
\draw [line width = 1.221582857518175, color=gray] (1.1488399227121215,-1) -- (1.1488399227121215,1);
\draw [line width = 1.4741707663789532, color=gray] (0.039615285950991286,-1) -- (0.039615285950991286,1);
\draw [line width = 1.2350909442516311, color=gray] (-1.1656692867346623,-1) -- (-1.1656692867346623,1);
\draw [line width = 1.1610487934649298, color=gray] (1.690472711420533,-1) -- (1.690472711420533,1);
\draw [line width = 2.279338868149255, color=gray] (2.6353006987253735,-1) -- (2.6353006987253735,1);
\draw [line width = 2.0245396550707806, color=gray] (0.7354325776006916,-1) -- (0.7354325776006916,1);
\draw [line width = 1.33191867986573, color=gray] (0.8111507813485321,-1) -- (0.8111507813485321,1);
\draw [line width = 2.7694869012113275, color=gray] (0.8513191683636929,-1) -- (0.8513191683636929,1);
\draw [line width = 1.2937072715587579, color=gray] (-0.5078388348863938,-1) -- (-0.5078388348863938,1);
\draw [line width = 1.2779169523049052, color=gray] (-1.2079794649692261,-1) -- (-1.2079794649692261,1);
\draw [line width = 2.5198832096829866, color=gray] (1.144318120844972,-1) -- (1.144318120844972,1);

\draw [line width = 1, dashed] (-2,1) --  (4,1);
\draw [line width = 1, dashed] (-2,-1) --  (4,-1);


\draw [line width = 1] (4.25,1)--(4.5,1);
\draw [line width = 1] (4.375,1)--(4.375,-1);
\draw [line width = 1] (4.25,-1)--(4.5,-1);
\node (a) at (5,0) {$C^u_{2l}(\hat{p})$};

\filldraw[fill=gray] (2,-1.5) circle (4pt);
\draw[fill=darkgray, line width = .0001] (2.5,.25) -- (3,.25) -- (3,.5) -- (2.5,.5) -- (2.5,.25);
\node (a) at (2.7,-1.5) {$=G_l$};

\draw [line width = 1] (-2,0) -- (4,0);
\draw [line width = 1] (0,-2) --  (0,4);
\filldraw (0,0) circle (1.5pt);
\node (a) at (-.2,-.25) {$\hat{p}$};
\node (a) at (-2.75,0) {$W^u_{loc}(\hat{p})$};
\node (a) at (-.8,4) {$W^s_{loc}(\hat{p})$};

\draw[fill=darkgray, line width = .0001] (2.5,.25) -- (3,.25) -- (3,.5) -- (2.5,.5) -- (2.5,.25);

\newcommand\xoff{6.1}
\newcommand\yoff{2.75}

\draw[fill=darkgray, line width = .0001] (\xoff -.25, \yoff+.8) -- (\xoff +.5, \yoff+.8) -- (\xoff +.5, \yoff+1.7) -- (\xoff -.25, \yoff + 1.7);
\filldraw (\xoff+.3, \yoff+1.25) circle (1.5pt);
\node (a) at (\xoff+.75, \yoff+1.25) {$\hat{q}$};
\node (a) at (\xoff-.8, \yoff+1.5) {$R_{\rho}(\hat{q})$};


\draw [line width = 1pt,->] (2.75,.4) arc(130:90:4 and 14);
\node (a) at (3.5,3) {$\hsigma^{k}$};

\end{tikzpicture}
\end{center}
\caption{Constructing $\hat{q}$ and $R_{\rho}(\hat{q})$.}
\end{figure}

Now we use the above lemma to construct a set of rectangles such that given a fixed hyperplane section, there exists some rectangle in the set along which the cocycle can map the unstable direction of a point in the rectangle to avoid the given hyperplane section. 

\begin{proposition}\label{prop:twisting_sets}
Fix $\gamma>0$ and $\theta_0$ any sufficiently small number. Then for every $l \in \N$ there exists $N\in \N$, $\hq\in \hSigma$, $\rho\in \N$, and a finite collection of points $\{\hq_i\}_{1\le i \leq N} \subseteq W^s_{loc}(\hq)$ and associated times $k_i \in \N$ such that:
\begin{enumerate}
    \item [(a)] $R_\rho(\hq_i) \cap R_\rho (\hq_j) = \varnothing$, for $i \neq j$;
    \item [(b)] $\hsigma^{-k_i}(R_\rho(\hq_i))\subseteq C^u_l(\hp)$, for each $1 \leq i \leq N$;
    \item [(c)] For any hyperplane section $H$ of $\Gr(k,d)$, there exists some $i \in \{1,\dots, N\}$ such that for any Oseledec regular $\hy \in \hsigma^{-k_i}(R_\rho(\hq_i))$ we have 
    \begin{equation}\label{eqn:twisting_transversal}  
  d_{\Gr(k,d)}(A^{k_i}(\hy)E^u(\hz), H) > \theta_0,
  \end{equation}
  where $\hz = [\hy, \hp]$.
  
  \item [(d)] For $\hy$ and $\hz$ as above, $A^{k_i}(\hy)\colon \Gr(k,d)\to \Gr(k,d)$ is $\gamma$-Lipschitz on a $\theta_0$-neighborhood of $E^u(\hz)$ 
\end{enumerate}
\end{proposition}

\begin{proof}

Let $\hat{q}$ be a point such that $\hat{\mu}^s_{\hq}$-a.e. point in $W^s_{loc}(\hat{q})$ is Osceledec regular and generic for $\hmu$. We will need the following claim, which we will prove immediately after this proposition for the sake of clarity:

\begin{claim} \label{claim: avoidhyperplanes} For any $\theta_0> 0$ sufficiently small, given any hyperplane section $H\subseteq \Gr(k,d)$, there exists a subset $S \subseteq W^s_{loc}(\hq)$ of positive $\mu^s_{\hq}$-measure such that for every $\hz \in S$:
$$d_{\Gr(k,d)}(E^u(\hz), H) > 2\theta_0.$$
\end{claim}

Now let $\theta_0$ be any sufficiently small number such that Claim \ref{claim: avoidhyperplanes} and Lemma \ref{lem:good_rectangles} hold. Again, observe that how small $\theta_0$ has to be does not depend on $l$.

Fix a collection of hyperplane sections $$\{H_1,\ldots,H_N\}$$ with the property that given any other hyperplane section $H$ there exists some $1 \leq i \leq N$ such that the Hausdorff distance in $\Gr(k,d)$ between $H$ and $H_i$ is less than $\theta_0/2$. This is possible since the space of hyperplane sections is compact with the Hausdorff distance. 

For each $1 \leq i \leq N$, we apply Claim \ref{claim: avoidhyperplanes} with $H = H_i$ to obtain a positive $\mu^s_{\hq}$-measure set $S_i \subseteq W^s_{loc}(\hq)$ of points $\hz$ with 
$$d_{\Gr(k,d)}(E^u(\hz), H_i) > 2\theta_0.$$
For each $1 \leq i \leq N$, we let $\hq_i$ be any such point in $S_i$ which is moreover generic for  and Oseledec regular for $\hmu$. Let $0<\epsilon<\theta_0$ and apply Lemma \ref{lem:good_rectangles} with this choice of $\epsilon,\theta_0$ 
 for each $\hq_i$ to obtain a sequence of rectangles $R_{\rho_i}(\hat{q}_i)$ satisfying the conclusions of that lemma. Let $\rho$ be greater maximum of the $\rho_i$ and chosen large enough that $R_\rho(\hq_i) \cap R_\rho (\hq_j) = \varnothing$, for $i \neq j$, verifying (a). It is evident since the rectangles $R_\rho(\hq_i)$ satisfy the conclusions of Lemma \ref{lem:good_rectangles} that (b) and (d) also hold for the constructed rectangles. It remains to verify (c). 

Let $H$ be any hyperplane section of $\Gr(k,d)$. By construction, there is some $1 \leq i \leq N$ such that the Hausdorff distance in $\Gr(k,d)$ between $H$ and $H_i$ is less than $\theta_0/2$. This, together with the fact that $\angle(E^u(\hat{q}_i),H_i)>2\theta_0$, and $$d_{\Gr}(A^{k_i}(E^u(\hat{z})),E^u(\hat{q_i})))<\theta_0,$$ by (2) in Lemma \ref{lem:good_rectangles} then gives (c) by the triangle inequality. 
\end{proof}

We now prove Claim \ref{claim: avoidhyperplanes}.
\begin{proof}[Proof of Claim \ref{claim: avoidhyperplanes}]

Let $\hat{m}$ be a $u$-state over $\hat{\mu}$ on $\hSigma \times \Gr(k,d)$. For a hyperplane section $H$, let $H_\theta$ be its $\theta$-closed neighborhood in $\Gr(k,d)$. Since by Proposition \ref{smoothconditionals} $m_x(H) = 0$ for any hyperplane section $H$, by upper semi-continuity for each $1 \leq i \leq N$ there is some $\theta_i> 0$ sufficiently small to ensure that $m_{q}(H_{\theta_i}) < 1/2$. 
By taking $\theta_0 = \min_{1 \leq i\leq N} \theta_i/3$, we see that:
\begin{equation} 
\label{eq: supustate} \sup_{1 \leq i \leq N} m_q(H^i_{3\theta_0}) < \frac{1}{2}.
\end{equation}

 Choose a set $\{H^i\}_{1 \leq i \leq N}$ of hyperplane sections of $\Gr(k,d)$ with the property that any other hyperplane section $H \subseteq H^i_{\theta_0/2}$ for some $1 \leq i \leq N$. This is possible since, in the Hausdorff topology, the space of hyperplane sections is compact and the set of hyperplane sections contained in a neighborhood of a fixed hyperplane section contains an open set.

By the disintegration formula:
\begin{align}
m_q &= \int_{W^s_{loc}(\hq)} \hm_{\hx}\, d\hmu^s_{\hq}(\hx)\\ 
&= \int_{W^s_{loc}(\hq)} \delta_{E^u(\hx)}\, d\hmu^s_{\hq}(\hx)\label{eqn:conditional_measures_Euk},\end{align}
 recalling that by Proposition \ref{prop: ustateslieondirac} for $\hmu^s_{\hp}$-a.e. $\hx$ in $W^s_{loc}(\hq)$, we have $\hm_{\hx} = \delta_{E^u(\hx)}$. 

 Now given an arbitrary hyperplane section $H$ by construction there is some $1 \leq i \leq N$ such that $H \subseteq H^i_{\theta_0/2}$. Then by equation (\ref{eq: supustate}) and the disintegration formula (\ref{eqn:conditional_measures_Euk}), we see that set $S$ of $\hx \in W^s_{loc}(\hq)$ with $E^u(\hx) \notin H^i_{3\theta_0}$, must have $\mu^s_{\hq}$-measure at least $1/2$. Finally, since $H \subseteq H^i_{\theta_0/2}$, we conclude that for $\hx \in S$, we have $d_{\Gr}(E^u(\hx), H) > 2\theta_0$, as desired.
\end{proof}

\begin{remark} Observe that $\theta_0$ does not depend on $l$ in the construction in Proposition \ref{prop:twisting_sets}. All other parameters, namely, $\hq'$, $\hq_i'$, $k_i$, $\rho$ and $N$, depend on $l$, but this dependence will not be relevant and we will often omit it.
\end{remark}

With the twisting construction done, now we move to the proof of convergence to a Dirac mass. The main idea is to construct for $\hmu$-a.e. $\hx$ a suitable subsequence of negative times along which the orbit of $\hx$ has some good properties which we define now. 

From now on, fix some $\theta_0$ sufficiently small so that Proposition \ref{prop:twisting_sets} holds and some $0 < \gamma < 1$. For each $l > 0$, we obtain a set of rectangles $R_\rho(\hq_i)$ as in Proposition \ref{prop:twisting_sets} for the parameters $\theta_0$, $l$ and $\gamma$, which for simplicity we will denote simply by $R_1^l, \dots R_N^l$. 

For each $\hx \in \hSigma$ and $n \in \N$, let $k \leq j \leq (d-1)$ be the smallest integer such that $\sigma_j(A^n(\hx_n))> \sigma_{j+1}(A^n(\hx_n))$. Then $S_{d-j}(\hx^n,n)$ is a well-defined $(d-j)$-dimensional linear subspace which defines some linear section $S$ on $\Gr(k, d)$. Fix $H_{\hx_n}$ to be any hyperplane section containing $S$. If $q$ does not exist, i.e. $\sigma_j(A^n(\hx^n)) = \sigma_{j+1}(A^n(\hx^n))$ for all $k \leq j \leq (d-1)$, we fix $H_{\hx^n}$ to be any hyperplane section of $\Gr(k, d)$. 

\begin{definition} We say that a point $\hx \in \hSigma$ is $l$-good if there exists a time $n_l > l$ such that:

\begin{enumerate}
    \item $\hx^{n_j} \in R_i^l$ for some $1 \leq i \leq N$.
    \item The conclusion of Proposition \ref{prop:twisting_sets}(c) with $H = H_{\hx^{n_j}}$ holds for the $R_i^l$ specified in the previous item. That is, for any Oseledec regular $\hy \in \hsigma^{-k_i}(R^l_i)$, we have
    $$ 
  d_{\Gr(k,d)}(A^{k_i}(\hy)E^u(\hz), H_{\hx^{n_j}}) > \theta_0,
  $$
  where $\hz = [\hy, \hp]$.
\end{enumerate}

\end{definition}

\begin{proposition}\label{prop:l_good_full_measure}
For each $l>0$ and any fully supported shift invariant probability measure $\heta$ on $\hSigma$ with continuous local product structure, the set of $l$-good points has full measure with respect to $\heta$.
\end{proposition}

Before proceeding to the proof we state a lemma. This lemma is an analogue of the Lebesgue density theorem for measures with local product structure. It is likely possible to deduce it from Federer's book \cite[Sec.~2.9]{federer1969geometric}, but we expect that this would take longer than reading the proof below.

\begin{lemma}\label{lem:measure_estimate}
Let $\hnu$ be a probability measure with local product structure on $\hat{\Sigma}$ and $0<\alpha<1$. Let $G\subseteq \hat{\Sigma}$ be some Borel measurable set. Suppose that for $\hnu$-almost every $\hat{x}\in \hat{\Sigma}$, that for $\hnu^s_{\hat{x}}$ a.e.~$\hy\in W^s_{loc}(\hat{x})$, that for any $\epsilon>0$ there exists a set $S\subset W^s_{loc}(\hat{x})\cap G$ and $0<\epsilon'\le \epsilon$, with 
\begin{enumerate}
    \item
    $S\subseteq B_{\epsilon'}(\hat{y})$,
 \item
    $$\frac{\hnu^s_{\hat{x}}(S)}{\hnu^s_{\hat{x}}(B_{\epsilon'}(\hat{x})\cap W^s_{loc}(\hat{x}))}>\alpha>0,$$
    \item
    $S\cap G$ is clopen in $W^s_{loc}(\hat{x})$.
\end{enumerate}
Then $\hat{\nu}(G)=1$.
\end{lemma}
\begin{proof}
Clearly from local product structure it suffices to verify that $\hat{\nu}^s_{\hat{x}}(G\cap W^s_{loc}(\hat{x}))=1$ for $\hat{\eta}$-a.e.\ $\hat{x}$.  We will embed $W^s_{loc}$ into Euclidean space and then make use of the following claim.
%

\

\noindent {\bf Claim:} For each $m\in \N$, there exists $C_m>0$ such that if $\mathbb{P}$ is a probability measure in $\R^m$ with compact support $S$ and $G\subseteq S$  is a set such that for $\mathbb{P}$-almost every $x\in S$ there exists arbitrarily small open balls $O(x)$ such that: 
\[
\frac{\P(G\cap O(x))}{\P(O(x))}>\alpha>0,
\]
then $\mu(G)\ge \alpha/C_m$. In fact, if each $G\cap O(x)$ contains a closed subset $K$ that is clopen in $S$ such that 
\[
\frac{\P(G\cap K)}{\P(O(x))}>\alpha>0,
\] then $G$ contains a closed subset of measure at least $\alpha/C_m$ that is clopen in $S$.
\begin{proof} Let $S'$ be the subset of full $\P$-measure for which the density property holds. Let $\mc{O} = \{O(x)\}_{x \in S'}$ be the open cover of $S'$ by the sets $O(x)$ guaranteed by the condition. Then the Besicovitch Covering Lemma implies there is some $C_m$ such that we may find $C_m$ families $\mc{A}_1,\ldots,\mc{A}_{C_m}\subseteq \mc{O}$ such that for each $1 \leq i\leq C_m$ the elements of $\mc{A}_i$ are disjoint and $S'$ is contained in the union of all the $\mc{A}_i$'s elements. Letting $A_i$ be the union of the sets in $\mc{A}_i$, we see that $\P(A_i\cap G)\ge \alpha\P(A_i)$.  Hence, as $\{A_i\}$ is an at most $C_m$-fold cover of $S'$, which has $\P(S') = 1$, it follows that $\P(G)\ge \alpha/C_n$. 

The claim about a closed subset of $G$ follows by taking the clopen in $S$ subsets $K_i$ in each $A_i$.
\end{proof}

The set $W^s_{loc}(\hat{x})$ is a Cantor set which we may embed H\"older continuously in $\{0,1\}^{\N}$ by coding the symbols used to define $\Sigma$. The Cantor set $\{0,1\}^{\N}$ has a H\"older continuous embedding in $\R$ as the middle thirds Cantor set by sending each string to the appropriate real number with no $1$'s in its ternary expansion. In fact, the embedding into $\R$ carries balls in $\{0,1\}^{\N}$ into balls in $\R$ because the embedding is increasing in the obvious way. Let $\phi$ denote this particular embedding $\phi\colon W^s_{loc}(\hat{x})\to \R$.

Let $\P = \phi_* \hnu^s_{\hx}$ denote the pushforward of the measure $\hat{\nu}^s_{\hx}$ on Euclidean space under the embedding $\phi$ and $G_\phi = \phi(G) \subseteq \Omega=\phi(W^s_{loc}(\hat{x}))$ denote the pushforward of the set of good points. This set is still Borel because it a Borel subset of a closed subspace. Moreover, as the embedding carries balls in the Cantor set into intervals in the real line, we may apply the claim, which requires a covering by balls. Thus we see that there is a closed subset $K_1\subset G$ with $\P(K_1)\ge \alpha/C_1$ that is clopen in $\Omega$. As $K_1$ is clopen in $\Omega$, $\Omega_2=\Omega\setminus K_1$ is a closed set and a positive distance away from $K_1$. Thus $\Omega_2$ also has a covering by balls as in the claim. And, importantly, because $\Omega_2$ is a positive distance from $K_1$, and the balls in the claim are arbitrarily small, from the claim we obtain another clopen subset $K_2\subseteq \Omega_2$ disjoint from $K_1$ with measure 
\[
\P(K_2)>(\alpha/C_1)\P(\Omega_2)\ge (\alpha/C_1)(1-\alpha/C_1).
\]
We may then consider the set $\Omega_2\setminus K_2$ and continue inductively. The $\P$-measure of the complement of $\sqcup_{i=1}^n K_i$ goes to zero exponentially quickly and the claim follows.
\end{proof}


We now return to the proof that $l$-good points have full measure.

\begin{proof}[Proof of Proposition \ref{prop:l_good_full_measure}]
Let $X_{l}$ denote the set of $l$-good points,
 and let $R_1,\ldots, R_N$ be the rectangles guaranteed by Proposition \ref{prop:twisting_sets} applied for the fixed $\theta_0>0$ and our choice of $l\in \N$. We proceed by checking the criterion in Lemma \ref{lem:measure_estimate} for $G=X_l$. The point of the proof is that if a point $\hx$ visits $R_i$ at some point $\hx^n$ in its past, then there is some $j$ such that the points in $W^s_{loc}(\hx^n)$ in $R_j$ are $\ell$-good. This is because $A^n(\hx^n)$ is constant on $W^s_{loc}(\hx^n)$ and this specific linear map is always $l$-good for the points in some $R_j$.

Let $R=\sqcup_{1\le i\le N} R_i$. The following is immediate from the continuous product structure: There exists $K>1$ such that for any $\hx\in R$ and $1\le i\le N$,
\begin{equation}\label{eqn:continuity_lower_bound_a}
K^{-1}<\hat{\eta}^s_{\hx}(R_i\cap W^s_{loc}(\hx))<K.
\end{equation}

For $\hat{\eta}$-a.e.\ $\hx\in \hat{\Sigma}$, there exist infinitely many $n\ge l$ such that $\hsigma^{-n}(\hx)\in R$. Consider now the matrix $A^n(\sigma^{-n}(\hx))$. Note that as the cocycle is constant on local stable leaves that for any $\hx'\in W^s_{loc}(\hsigma^{-n}(\hx))$ that $A^n(\hx')=A^n(\hsigma^{-n}(\hat{\omega}))$. By Proposition \ref{prop:twisting_sets}, there exists some $i\in \{1,\ldots, N\}$ such that if we take $H=S_{d-d^u}(A^n(\hsigma^{-n}(\hx)))$, and any $\hat{y}$ in $R_i\cap W^s_{loc}(\hat{\omega}')$, then \eqref{eqn:twisting_transversal} holds for that $\hat{y}$. This means that $\hsigma^n(\hat{y})$ is $\ell$-good at time $n$.

In particular, we see that all points in $\hsigma^n(R_i\cap W^s_{loc}(\hsigma^{-n}(\hx)))$ are in $X_l$. Call this set $G_{n,\hx}$. In order to conclude, it suffices to show that there exists some $D>0$ independent of $\hx$ such that 
\[
\frac{\hat{\eta}^s_{\hx}(G_{n,\hx})}{\hat{\eta}^s_{\hx}(\hsigma^n(W^s_{loc}(G_{n,\hx})))}>D>0.
\]
But by equation \eqref{eqn:continuity_lower_bound_a}, this is immediate from Lemma \ref{lem:measure_distortion_local_stable}. We are now done because we verified the criterion in Lemma \ref{lem:measure_estimate}.
\end{proof}

Finally, we show that the conditional measures for a $u$-state of $\heta$ must also be Dirac masses, i.e. Proposition \ref{prop: convergencetodirac}. The proof will use the following lemma:
   \begin{lemma} \label{lem: equicontinuity} Let $\varepsilon > 0$ and $\theta_0 > 0$. Then there exists $\delta > 0$ and  such that for any $A \in \GL(d,\R)$ the following holds. Fix $1 \leq k \leq d$ and suppose that $q \geq k$ is the smallest integer such that $\sigma_q(A) > \sigma_{q+1}(A).$ Let $H$ be any geometric hyperplane section of $\Gr(k, d)$ containing the linear section defined by $S_q(A)$ in $\Gr(k,d)$, and if $q$ does not exist (i.e. $\sigma_q(A) = \sigma_{q+1}(A)$ for $k \leq q \leq d-1$), let $H$ be any hyperplane section. 

    Then for any $\xi_1, \xi_2 \in \Gr(k, d)$ such that $d_{\Gr(k,d)}(\xi_i, H) > \theta_0$, $i = 1,2$:
    $$d_{\Gr(k,d)}(\xi_1, \xi_2) < \delta \implies d_{\Gr(k,d)}(A\xi_1, A\xi_2) < \varepsilon. $$
    \end{lemma}

    \begin{proof}
Let $A \in \GL(d,\R)$. Since the metric on the Grassmannian is invariant by isometries of $\R^d$, by the singular value decomposition theorem we may assume that $A$ is of the form $A = \text{diag}(a_1, \dots, a_d)$, with $a_1 \geq \dots \geq a_d > 0$. Without loss of generality we may also assume that $H$ is the hyperplane section defined by $e_{k+1} \wedge \dots \wedge e_d$. 

Thus, we can describe elements of $\{\xi \in \Gr(k,d): d_{\Gr(k,d)}(\xi, H) > \theta_0\}$ as graphs with bounded slopes over the first $k$ coordinates, i.e., we may embed the preceeding set into the set $\text{End}_M(\R^k, \R^{d-k})$ of linear maps from the first $k$ coordinates to the last $d-k$ coordinates of $\R^d$ with a norm bounded by $M > 0$, where $M$ only depends on $\theta_0$. Moreover, it is not hard to see that in this set $d_{\Gr(k,d)}$ is equivalent to the distance induced by the operator norm on $\text{End}_M(\R^k, \R^{d-k})$. Since all matrix norms are equivalent we endow $\text{End}_M(\R^k, \R^{d-k})$ with the max norm, i.e., for $\phi \in \text{End}_M(\R^k, \R^{d-k})$ given in a basis by $\phi_{ij}$, we let $\|\phi\| = \max_{i,j} |\phi_{ij}|$ . 

Let $\phi_1, \phi_2 \in \text{End}_M(\R^k, \R^{d-k})$ correspond to $\xi_1, \xi_2 \in \Gr(k,d)$ as in the hypotheses of the theorem and $A\phi_1, A\phi_2 \in \text{End}_M(\R^k, \R^{d-k})$ correspond to $A\xi_1, A\xi_2 \in \Gr(k,d)$. To conclude, it suffices to find a uniform bound for $|(A\phi_1-A\phi_2)_{ij}|$ in terms of $|(\phi_1-\phi_2)_{ij}|$. It is clear that for any $\phi \in \text{End}_M(\R^k, \R^{d-k})$, we have $|(A\phi)_{ij}| \leq (a_{k+1}/a_k) \dot |\phi_{ij}|$, since the action of $A$ expands the first $k$ coordinates by at least $a_k$ and the last $d-k$ by at most $a_{k+1}$. Since $a_{k+1}/a_k \leq 1$ the proof is complete. 
\end{proof}

    \begin{remark}
        Observe that the crucial fact is that $\delta$ works for \textit{any} $A \in \GL(d,\R).$ In other words, the action of linear maps in a region uniformly bounded away from the least expanded subspaces is equicontinuous.
    \end{remark}
\begin{proof}[Proof of Proposition \ref{prop: convergencetodirac}]
    Since the countable intersection of full measure sets has full measure, by Proposition \ref{prop:l_good_full_measure}, we see that $\heta$-a.e. $\hx$ is $l$-good for every $l$. Recall that for $\heta$-a.e. $\hx$ by Proposition \ref{prop: martingale}:
    $$\hm^{\eta}_{\hx} = \lim_{n \to \infty} A^n(\hx^n)_* m^{\eta}_{\hat{x}^n}.$$ 

    Hence there is a full $\heta$-measure set $S \subseteq \hSigma$ such that the above holds and every point in $S$ is $l$-good for every $l > 0$. Fix any such $\hx$ in $S$.  For all $l> 0$  fix $n_l > l$ to be a time guaranteed by the definition of $l$-good for $\hx$. We say that a probability measure $m$ on a metric space is $\varepsilon$\emph{-concentrated} at a point $z$ if $m(B_{\epsilon}(z))\ge 1-\varepsilon$, where $B_{\epsilon}(z)$ is the open ball of radius $\epsilon$ about $z$.

\begin{claim}\label{claim:eps_concentrated}
Let $\varepsilon > 0$. Then for all $l > 0$ sufficiently large there is an $i \in \{1, \dots ,N\}$ such that for $k_i = k_i(l)$ as in Proposition \ref{prop:twisting_sets}, the measure
$$A^{n_l+k_i+l}(\hx^{n_l+k_i+l})_* m^{\eta}_{x^{n_l+k_i+l}}$$ is $\varepsilon$-concentrated.
\end{claim}

    We will now show how to conclude given the claim. Since $n_l > l$, by passing to a subsequence of $n_l$ we may assume that $n_{l+1} > n_l+k_i(l)+l$, where $k_i(l)$ is as in the claim, and that the point around which the masses $\varepsilon$-concentrate converges to some $\xi^u(\hx) \in \Gr(k, d)$, by compactness of the Grassmannian, from which we conclude that
    $$\lim_{l \to \infty} A^{n_l+k_i+l}(\hx^{n_l+k_i+l})_* m^{\eta}_{x^{n_l+k_i+l}} = \delta_{\xi^u(\hx)},$$
    and the proof will be completed. Hence we are done once we prove the claim about $\varepsilon$-concentration of masses.

\begin{proof}[Proof of Claim \ref{claim:eps_concentrated}.]

We begin my fixing some constants. Fix some small $\varepsilon>0$. From Proposition \ref{prop:twisting_sets}, there exists a fixed $\theta_0>0$ such that the conclusions of that proposition hold.
For each $l$, we let $R_i^{l}$ denote the collection of rectangles obtained from Proposition \ref{prop:twisting_sets} applied with that fixed choice of $\theta_0>0$.
     By definition of $n_l$, $\hx^{n_l} \in R_i^l$ for some $1 \leq i \leq N$ and a rectangle obtained from Proposition \ref{prop:twisting_sets}. Hence $\hx^{n_l + k_i} \in C^u_l(\hp)$, where $k_i$ is also given by Proposition \ref{prop:twisting_sets}.

By Lemma \ref{lem: equicontinuity}, given this choice of $\varepsilon>0$ and $\theta_0>0$ there some exists $\delta_1>0$ such that if $m$ is a measure on $\Gr(k,d)$, $A\in \GL(d,\R)$, and $m$ is $\delta_1$ concentrated at a point $V\in \Gr(k,d)$ that is $\theta_0$ away from $H$, a geometric hyperplane defined by a subspace containing the most contracting subspace $S_q(A)$ of $A$ as defined in Lemma \ref{lem: equicontinuity}, then $A_*m$ is $\varepsilon$-concentrated. Note that this $\delta_1$ is not literally the $\delta$ from that lemma but may be smaller to ensure that the conclusion of the lemma applies to the set where $m$ is concentrated.

Let $\alpha=\min\{\theta_0/2,\delta_1\}>0$.

We will show $\varepsilon$-concentration by breaking down $$A^{n_l+k_i+l}(\hx_{n_l+k_i+l})_* m^{\eta}_{x_{n_l+k_i+l}}$$ into
\[A^{n_l}(\hx^{n_l})_*A^{k_i}(\hx^{n_l+k_i})_*A^{l}(\hx^{n_l+k_i+l})_* m^{\eta}_{x^{n_l+k_i+l}},\]
    and pushing the measure through each of these terms in order. We give each of these terms names:
    \begin{align}
    m_1&=A^l(\hx^{n_l+k_i+l)})_*m^{\eta}_{x^{n_l+k_i+l}}\\
    m_2&=A^{k_i}(\hx^{n_l+k_i})_*m_1\\
    m_3&=A^{n_l}(\hx^{n_l})_*m_2.
    \end{align}


The proof then breaks down into a sequence of three claims, which follow from our choice of constants at the beginning of the proof.

\vspace{.2cm}
   \noindent {\bf Subclaim 1.} The measure $m_1$ is $\alpha$-concentrated around $E^u(\hz)$, where $\hz=W^s_{loc}(\hx^{n_l+k_i})\cap W^u_{loc}(\hp)$.
    \begin{proof}
       By definition $\hat{z}\in W^s_{loc}(\hat{x}^{n_l+k_i})$. But note that as $n_l$ is an $l$-good time that $\hat{x}^{n_l+k_i}\in C^u_l(\hp)$. This implies that $\hat{x}^{n_l+k_i+l}$ and $\hat{z}^l$ lie in the same local stable leaf. Thus as the cocycle is constant along local stable leaves,
    \begin{equation}\label{eqn:equal_pasts}
A^{l}(\hx^{n_l+k_i+l})_* m^{\eta}_{x^{n_l+k_i+l}} = A^{l}(\hz^l)_* m^{\eta}_{z^l},
    \end{equation}
hence it suffices to show that the right hand side of \eqref{eqn:equal_pasts} is $\alpha$-concentrated around $E^u(\hat{z})$. To see this, note that 
that
    \[
    A^{l}(\hz^l)_* m^{\eta}_{z^l} = (A^{-l}(\hz)^{-1} A^{-l}(\hp))_* A^{l}(\hp)_* m^{\eta}_{z^l}.
    \]
     The expression in parenthesis in the previous line converges to $H^u_{\hp \hz}$ as $l \to \infty$. Moreover, $A^{l}(\hp)_* m^{\eta}_{z_l}$ converges to $\delta_{E^u(\hp)}$ since $m^{\eta}_{z_l}$ converges weak* to $m^{\eta}_{\hp}$ and the latter gives 0 measure to hyperplane sections. Hence by Lemma \ref{lem: quasiprojectiveconvergence }, we see that the right-hand side of \eqref{eqn:equal_pasts} converges to $\delta_{H^u_{\hp\hz}E^u(\hp)}$.

     Therefore, for $l$ sufficiently large it follows that $$m_1 := A^{l}(\hz_l)_* m^{\eta}_{z_l}$$ is $\alpha$-concentrated around $H^u_{\hp \hz} E^u(\hp) = E^u(\hz)$ (Lemma \ref{lem:unstable_bundles_cts_along_unstable_leaves}), as desired. 
    \end{proof}

    \noindent {\bf Subclaim 2.} The measure $m_2$ is $\delta_1$-concentrated at a point at least $\theta_0$ away from $H=H_{\hx^{n_l}}$, where $H$ is defined as in the definition of $\ell$-good.
    \begin{proof}
    From Subclaim 1, $m_1$ is $\alpha$-concentrated at $E^u(\hz)$. The application of Proposition 4.4 gives that 
    \[
    d_{\Gr}(A^{k_i}(\hx^{n_l+k_i})E^u(\hz),H_{\hx^{n_l}})>\theta_0,
    \]
    and that $A^{k_i}(\hx^{n_l+k_i})$ is $1$-Lipschitz in a $\theta_0$ neighborhood of $E^u(\hz)$. Thus as $\alpha<\theta_0$, by the $1$-Lipschitzness $m_2$ is $\alpha$-concentrated at a point $\theta_0$ away from $H$. As $\alpha\le \delta_1$ we are done.
    \end{proof}
    
    \noindent {\bf Subclaim 3.} The measure $m_3$ is $\varepsilon$-concentrated.
    \begin{proof}
    From the conclusion of Subclaim 2, $H$ is geometric hyperplane defined by a subspace containing the appropriate most contracting subspace of $A^{n_l}(\hx^{n_l})$ specified by Lemma \ref{lem: equicontinuity}. Thus by our choice of $\delta_1$ in the first paragraph,
    as $m_2$ is $\delta_1$-concentrated $\theta_0$ away from $H$, it follows that $m_3$ is $\varepsilon$-concentrated.
    \end{proof}
    
    We are now done with the proof of Claim \ref{claim:eps_concentrated} as Subclaim 3 is the needed conclusion.
\end{proof}

Having finished the proof of Claim \ref{claim:eps_concentrated}, we obtain the proof of Proposition \ref{prop: convergencetodirac} by the discussion after the statement of that claim.
\end{proof}

\section{Conclusion of Proof of Theorem \ref{theorem: maintheorem}} \label{sec: proof}
In Proposition \ref{prop: convergencetodirac} from the previous section we obtained a measurable $\xi^u\colon \hSigma \to \Gr(k, d)$ defined $\heta$ almost everywhere. This section satisfies the same properties as the map $\xi^u$ considered in \cite[Proposition 6.4]{avila2007simplicity} where it is shown to have the following properties for $\heta$-almost every $\hx \in \hSigma$:
\begin{enumerate}
    \item [(i)] $\xi^u$ is invariant under the cocycle and the unstable holonomies.
    \item [(ii)] $\sigma_{k}(\hA^n(\hsigma^{-n}(\hx)))/\sigma_{k+1}(\hA^n(\hsigma^{-n}(\hx))) \to \infty$, and the image of the $k$-dimensional subspace most expanded by $\hA^n(\hsigma^{-n}(\hx))$ converges to $\xi^u(\hat{x})$. 
\end{enumerate} 

The idea of the rest of the proof of Theorem \ref{theorem: maintheorem} is to show that the section $\xi^u$ is the sum of the $k$ strictly largest Osceledec subbundles for $\hA$ with respect to $\heta$. To do so, we first obtain a $\heta$-almost everywhere  defined section $\xi^T\colon \hSigma \to \Gr(k, d)$ by the same argument as in the proof of Proposition \ref{prop: convergencetodirac}, whose orthogonal complement $\xi^s$ is a candidate sum of the $(d-k)$ strictly smallest Osceledec subbundles for $\hA$ with respect to $\heta$. Finally, we show that indeed the Lyapunov exponents along $\xi^u$ are strictly greater than those of $\xi^s$.

To obtain $\xi^T$, we apply Proposition \ref{prop: convergencetodirac} to the \textit{adjoint cocycle} to $\hA$, which is defined by $\hat{B}(\hx) := \hA(\hsigma^{-1}(\hx))^T$, where the $T$ superscript denotes transpose with respect to a fixed inner product on $\R^d$. By definition, the adjoint cocycle $\hat{B}$ is a cocycle over $\hsigma^{-1}$. It is easily verified \cite[Section 7.1]{avila2007simplicity} that the adjoint cocycle is also fiber-bunched and satisfies the integrability hypothesis of Oseledec' theorem. From the assumption $\lambda_{k}(\hA, \hmu) > \lambda_{k+1}(\hA, \hmu)$, it is easy to see that $\lambda_{k}(\hB, \hmu) > \lambda_{k+1}(\hB, \hmu)$ since $\hA$ and $\hB$ have the same singular values. Moreover:
\begin{claim} For any $1 \leq k \leq d-1$, the cocycle $\hA$ is strongly irreducible on $\Gr(k, d)$ if and only if the adjoint cocycle $\hB$ is strongly irreducible on $\Gr(d-k, d)$.
\end{claim}
\begin{proof} By duality it suffices to show one of the directions, so we prove that $\hA$ is strongly irreducible on $\Gr(k, d)$ if $\hB$ is strongly irreducible on $\Gr(d-k, d)$.

Suppose for contradiction that $\hB$ is not strongly irreducible on $\Gr(d-k, d)$, i.e., that there exists a non-trivial continuous family of $(d-k)$-linear arrangements $\cL$ invariant by $\hB$. We define $\cL^T$ to be a continuous family of $k$-linear arrangements given by $\cL^T = \star \cL$, where $\star$ is the usual Hodge star map  $\Lambda^{d-k} \R^d \to \Lambda^{k} \R^d$ restricted to the Grassmannian. Since the Hodge star map is linear, we see that $\cL^T$ is by definition a continuous family of $k$-linear arrangements, and since $\star$ sends rank 1 vectors to rank 1 vectors $\cL^T$ must be non-trivial. 

Moreover, since restricted to the Grassmannians the Hodge star maps subspaces to their orthogonal complements, and the transpose cocycle has the property that $\star A\xi = A^T \star \xi$ for all subspaces $\xi$, we see that $\cL^T$ must be invariant by $\hA$. This contradicts strong irreducibility of $\hA$ on $\Gr(k, d)$.
\end{proof}

Therefore we may apply Proposition \ref{prop: convergencetodirac} to the cocycle $\hB$ and obtain a section $\xi^T\colon \hSigma \to \Gr(k, d)$ and let $\xi^s\colon \hSigma \to \Gr(d-k, d)$ be obtained by taking the orthogonal complement of $\xi^T$. The proof of the Lyapunov exponent gaps of $\xi^u$ and $\xi^s$ now follows in the exact same way as the proof of the main theorem in \cite[Section 7.2]{avila2007simplicity}. For the reader's convenience we include an outline below:

\begin{proof}[Outline of proof of Theorem \ref{theorem: maintheorem}] The proof is broken down into several steps:

\vspace{.2cm}
\textbf{Step 1:} $\xi^u(\hx)$ and $\xi^s(\hx)$ intersect transversely for $\heta$-almost every $\hx$. 
\begin{proof}
    Recall that for the reduced cocycle the stable holonomies are trivial. Hence for $\heta$-almost every $\hx$, $\xi^s$ is constant on $W^s_{loc}(\hx)$. Fix one such $\hx$ and let $H$ be the hyperplane section defined by $\xi^s$ on $\Gr(k, d)$. By Proposition \ref{smoothconditionals} and Proposition \ref{prop: ustateslieondirac}, the distribution of $\xi^u$ on $W^s_{loc}(\hx)$ avoids a fixed hyperplane section $\heta^s_{\hx}$-almost everywhere. In particular it avoids $H$, so that $\xi^u$ and $\xi^s$ intersect transversely $\heta^s_{\hx}$-almost everywhere on $W^s_{loc}(\hx)$. Since this holds for $\heta$-almost every $\hx$, the proof is complete.
\end{proof}

Thus we have $\heta$-a.e.~defined $\hat{A}$ invariant transverse sections $\xi^u$ and $\xi^s$, so each of these bundles has a splitting into subspaces at $\hat{\eta}$-a.e.~point. Let $\xi_1 \subseteq \xi^u$ be the bundle corresponding to the weakest Lyapunov subspace in $\xi^u$ and let $\xi_2 \subseteq \xi^s$ be the bundle corresponding to the strongest Lyapunov subspace in $\xi^s$. Let $d_1, d_2$ be the dimensions and $\lambda_1, \lambda_2$ be the exponents of $\xi_1$ and $\xi_2$. The proof will be completed if we show $\lambda_1 > \lambda_2$.

Define for $n \geq 0$
$$\Delta^n(\hx) = \frac{\det(\hA^n(\hx), \xi_1(\hx))^{1/d_1}}{\det(\hA^n(\hx), \xi_1(\hx)\oplus \xi_2(\hx))^{1/(d_1+d_2)}},$$
where $\det(A, V)$ denotes the determinant of $A$ restricted to $V$. Then by Oseledec' theorem:
$$\lim_{n \to \infty} \frac{1}{n}\log  \Delta^n(\hx) = \frac{d_2}{d_1+d_2}(\lambda_1 - \lambda_2),$$
so it suffices to show that $ \log \Delta^n(\hx)$ grows linearly. 

\vspace{.2cm}
\textbf{Step 2:} For $\heta$ almost every $\hx$
\begin{equation}\label{eqn:diverges_to_infty}
\lim_{n \to \infty} \Delta^n(\hx) = +\infty.
\end{equation}
\begin{proof}[Proof sketch.] The idea of the proof is to use the defining property of $\xi^u$ which we mentioned in the beginning of the section, namely that $$\sigma_{k}(\hA^n(\hsigma^{-n}(\hx)))/\sigma_{k+1}(\hA^n(\hsigma^{-n}(\hx))) \to \infty,$$ and that the image of the $k$-dimensional subspace most expanded  by $\hA^n(\hsigma^{-n}(\hx))$ converges to $\xi^u(\hat{x})$. This, combined with the transversality statement of Step 1, suffices to show divergence of the ratio of the determinants in the definition of $\Delta^n$ and obtain Equation \eqref{eqn:diverges_to_infty}. For details, see \cite[Proposition 7.3]{avila2007simplicity}. 
\end{proof}

\vspace{.3cm}

\textbf{Step 3:} To conclude one uses the following well-known result. Let $T\colon X \to X$ be a measurable transformation preserving a probability measure $\nu$ in $X$, and $\phi\colon X \to  \R$ be a $\nu$-integrable function such that $$\lim_{n \to \infty} \sum_{j = 0}^n (\phi \circ T^j)(x) = \infty$$ at $\nu$-almost every point. Then $\int_X \phi \, d\nu > 0$. 

The proof is then finished by applying the statement above to $\log \Delta(x)$ and applying the Birkhoff ergodic theorem.
\end{proof}

\section{Furstenberg's Theorem} \label{sec: furstenberg}

In this section we show an analog of Furstenberg's theorem in the fiber bunched setting. Compare the following theorem with \cite[Thm.~1]{bonatti2003}, which offers a similar alternative between every measure being hyperbolic and there existing a holonomy invariant measure on the projective extension of the cocycle. The measure obtained in that theorem represents the flags obtained in the proof below.

\begingroup
\def\thetheorem{\ref{thm:furstenberg_analog}}
\begin{theorem}
Suppose that $\hat{\Sigma}$ is a transitive hyperbolic system and that $\hat{A}\colon \Sigma\to \GL(d,\R)$ is a fiber-bunched H\"older continuous cocycle that is strongly irreducible on $\Gr(1,d)$. Then exactly one of the following holds:
\begin{enumerate}
\item
For every measure $\hat{\mu}$ with full support and continuous product structure $\hat{\mu}$ has two distinct Lyapunov exponents, or;
\item
$\hat{A}$ preserves a H\"older continuous conformal structure, and thus every invariant measure has all Lyapunov exponents equal to $0$.
\end{enumerate}
\end{theorem}
\addtocounter{theorem}{-1}
\endgroup

\begin{proof}
Suppose that there exists some other such measure that has only one Lyapunov exponent. Then by \cite[Thm.~3.4]{kalinin2013cocycles}, there exists a finite cover $\wt{\mc{E}}$ of the cocycle on which we have that the cocycle preserves a continuous flag 
\[
\{0\}=\wt{\mc{E}}^0\subset \cdots \subset \wt{\mc{E}}^k=\wt{\mc{E}},
\]
and continuous invariant conformal structures on each quotient $\wt{\mc{E}}^{i+1}/\wt{\mc{E}}^i$. This flag projects down to a collection of invariant subspaces contained in the original bundle $\mc{E}$. Hence strong irreducibility implies that $\wt{\mc{E}}^1=\wt{\mc{E}}^k$, so the original cocycle preserves a conformal structure. Hence all invariant measures on $\hSigma$ have a single Lyapunov exponent and the result follows.
\end{proof}

Note that the above theorem gives no information on whether there are more than two exponents.

\bibliographystyle{amsalpha}
\bibliography{biblio.bib}

\end{document}